\documentclass[reqno,11pt]{amsart}

\usepackage{amsmath,amssymb,amsthm,mathrsfs}
\usepackage{epsfig,color}
\usepackage[latin1]{inputenc}

\usepackage[numbers]{natbib}
\usepackage{dsfont}

 \allowdisplaybreaks
\numberwithin{equation}{section}

\def\H{\mathcal H}

\def\R{\mathbb R}
\def\N{\mathbb N}

\def\e{\varepsilon}
\def\s{\sigma}

\def\vphi{\varphi}
\def\Div{{\rm div}\,}
\def\om{\omega}

\def\g{\gamma}

\def\Om{\Omega}
\def\de{\delta}
\def\Id{{\rm Id}}

\def\spt{{\rm spt}}

\def\pa{\partial}
\def\trace{{\rm tr}}

\def\G{\Gamma}
\def\Hi{\mathcal{H}}

\newcommand{\vol}{\mathrm{vol}\,}

\newcommand{\n}{\nabla}

\renewcommand{\Div}{{\rm div \,}}
\newcommand{\ov}{\overline}

\newcommand{\DIV}{{\rm Div \,}}

\newcommand{\cc}{\subset\subset}

\newcommand{\dett}{{\rm det}}

\def\weak{\stackrel{*}{\rightharpoonup}}

\def\B{\mathcal{B}}

\def\ee{\mathrm{e}_1}
\def\ind{1}

\theoremstyle{plain}
\newtheorem{theorem}{Theorem}[section]
\newtheorem{lemma}[theorem]{Lemma}
\newtheorem{corollary}[theorem]{Corollary}
\newtheorem{proposition}[theorem]{Proposition}

\newtheorem{remark}[theorem]{Remark}

\def\U{\mathcal{U}}

\def\psharp{{p^\#}}
\def\pstar{{p^\star}}
\def\ISO{{\rm ISO}\,}

\def\supp{{\rm spt}}

\def\Ha{{H}}

\def\He{H_{\e}}

\newcommand{\norm}{{L^\pstar(\Ha)}}
\newcommand{\pnorm}{{L^p(\Ha)}}
\newcommand{\tnorm}{{L^\psharp(\pa \Ha)}}

\title[A bridge between Sobolev and Escobar inequalities]{A bridge between Sobolev and Escobar inequalities \\ and beyond}

\author[Maggi]{Francesco Maggi}
\address[Francesco Maggi]{
\newline \indent Abdus Salam International Center for Theoretical Physics,
\newline \indent Strada Costiera 11, I-34151, Trieste, Italy.
\newline \indent On leave from the University of Texas at Austin}
\email{fmaggi@ictp.it}

\author[Neumayer]{Robin Neumayer}
\address[Robin Neumayer]{ \newline
\indent  Department of Mathematics, The University of Texas at Austin,  \newline
\indent 2515 Speedway Stop C1200, Austin, TX 78712, USA}
\email{rneumayer@math.utexas.edu}

\begin{document}

\maketitle

\begin{abstract}
The classical Sobolev and Escobar inequalities are embedded into the same one-parameter family of sharp trace-Sobolev inequalities on half-spaces. Equality cases are characterized for each inequality in this family by tweaking a well-known mass transportation argument and lead to a new comparison theorem for trace Sobolev inequalities. The case $p=2$ corresponds to a family of variational problems on conformally flat metrics which was previously settled by Carlen and Loss with their method of competing symmetries. In this case minimizers interpolate between conformally flat spherical and hyperbolic geometries, passing through the Euclidean geometry defined by the fundamental solution of the Laplacian.
\end{abstract}

\section{Introduction} \subsection{A variational problem interpolating the Sobolev and Escobar inequalities}\label{overview} The goal of this paper is to illustrate a strong link between the {\it Sobolev inequality} on $\R^n$
\begin{equation}
  \label{sobolev inequality lp}
  \|\nabla u\|_{L^p(\R^n)}\ge S\,\|u\|_{L^\pstar(\R^n)}\qquad \pstar=\frac{np}{n-p}\,,
\end{equation}
and the {\it Escobar inequality} on the half-space $H=\{x_1>0\}$
\begin{equation}
  \label{escobar inequality lp}
  \|\nabla u\|_{L^p(H)}\ge E\,\|u\|_{L^{\psharp}(\pa H)}\qquad\psharp=\frac{(n-1)p}{n-p}\,,
\end{equation}
where $n\ge2$ and $p\in[1,n)$.
These classical sharp inequalities both arise as particular cases of the variational problem
\begin{eqnarray}\label{phi p T}
  \Phi(T)=\Phi^{(p)}(T)=\inf\Big\{\|\nabla u\|_{L^p(H)}:\|u\|_{L^\pstar(H)}=1\,,\|u\|_{L^\psharp(\pa H)}=T\Big\}\qquad T\ge0\,,
\end{eqnarray}
with $T=0$ in the case of \eqref{sobolev inequality lp}, and with $T=T_E$ for a suitable $T_E>0$ in the case of \eqref{escobar inequality lp}. Our main result (consisting of Theorems \ref{thm: main} and \ref{thm: properties of Phi} below) characterizes the minimizers of $\Phi(T)$ for every $T>0$ and allows one to describe the behavior of $\Phi(T)$ for every $T>0$.

The cases $p=2$ and $p=1$ have interpretations in conformal geometry and in capillarity theory respectively. In particular, when $p=2$, \eqref{phi p T} amounts to minimizing a total curvature functional among conformally flat metrics on $H$ -- see \eqref{PsiP} below. An interesting feature of this problem is that the corresponding minimizing geometries change their character from spherical (for $T\in(0,T_E)$) to hyperbolic (for $T>T_E$).

The characterization of minimizers in \eqref{phi p T} when $p=2$  is due to Carlen and Loss. In \cite{carlenloss} they deduce this result by an elegant application of their method of competing symmetries. The competing symmetries in play are: (1) the operation of spherical decreasing rearrangement, and (2) the composition of a pull-back by inverse stereographic projection from $\R^n$ to the $n$-dimensional sphere $\mathbb{S}^n$, a rotation on $\mathbb{S}^n$, and a final push-forward by stereographic projection. The use of conformal invariance seems to pose a non-trivial obstacle to the applicability of this approach when $p\ne 2$. From this point of view, our use of mass transportation for extending the Carlen--Loss result to the full range $p\in(1,n)$ seems appropriate.

Let us start by setting our terminology and framework, focusing on the case $p\in(1,n)$. We work with locally summable functions $u\in L^1_{{\rm loc}}(\R^n)$ that are vanishing at infinity, that is, $|\{|u|>t\}|<\infty$ for every $t>0$. If $Du$ denotes the distributional gradient of $u$, then the minimization in \eqref{phi p T} is over functions with $Du=\nabla u\,dx$ for $\nabla u\in L^p(H;\R^n)$. For every such function, the Sobolev inequality \eqref{sobolev inequality lp} implies that $u\in L^\pstar(\R^n)$. The constant $S=S(n,p)>0$ appearing in \eqref{sobolev inequality lp} is, by definition, the largest possible constant. It can be computed by exploiting the fact, proven in \cite{aubin,talenti}, that equality holds in \eqref{sobolev inequality lp} if and only if there exist $\lambda>0$ and $z\in\R^n$ such that
\begin{equation}
  \label{sobolev equality}
  u(x)=\lambda^{(n-p)/p}\,U_S(\lambda(x-z))\qquad \forall x\in\R^n\,,
\end{equation}
where
\begin{equation}\label{eqn: US}
U_S(x)=(1+|x|^{p'})^{(p-n)/p}\qquad x\in\R^n\,.
\end{equation}
(Here, as usual, $p'=p/(p-1)$.) The Escobar inequality has a similar meaning, with $H$ replacing $\R^n$. Again, $E=E(n,p)>0$ denotes the largest admissible constant in \eqref{escobar inequality lp}. Equality holds in \eqref{escobar inequality lp} if and only if there exist $\lambda>0$ and $z\in \R^n$ with $z_1< 0$ such that
\begin{equation}
  \label{escobar equality}
  u(x)=\lambda^{(n-p)/p}\,U_E(\lambda(x-z))\qquad \forall x\in H\,,
\end{equation}
where $U_E$ is the fundamental solution of the $p$-Laplacian on $\R^n$:
\begin{equation}\label{eqn: UE}
U_E(x)=|x|^{(p-n)/(p-1)}\,,\qquad x\in\R^n\setminus\{0\}\,.	
\end{equation}
The Escobar inequality was proven, together with its characterization result, in \cite{Escobar1988} for $p=2$; see also \cite{Beckner93}. In \cite{Nazaret2006}, by means of mass transportation techniques, the inequality is proven for every $p\in(1,n)$, along with the optimality of $U_E$. The characterization result is still missing in \cite{Nazaret2006}, but this only depends on some minor technical points that are  filled here.

Referring to the monograph \cite{mazyaBOOK} for a broader picture on Sobolev-type inequalities, we now pass to the starting point of our analysis, which is the realization that \eqref{sobolev inequality lp} and \eqref{escobar inequality lp} can be ``embedded'' in the family of variational problems \eqref{phi p T}. Indeed:

\medskip

\noindent (a) The Sobolev inequality is essentially equivalent to the variational problem $\Phi(T)$ with the choice $T=0$. Indeed, if $u=0$ on $\pa H$, then by applying \eqref{sobolev inequality lp} to the zero extension of $u$ outside of $H$, we find that $\Phi(0)\ge S$. Next, by considering an appropriate sequence of scalings as in \eqref{sobolev equality} multiplied by smooth cutoff functions, we actually find that
\[
\Phi(0)=S\,.
\]
The characterization of equality cases in \eqref{sobolev inequality lp} implies that $\Phi(0)$ does not admit minimizers. However, a concentration-compactness argument shows that every minimizing sequence is asymptotically close to a sequence of optimal functions in the Sobolev inequality that is either concentrating at an interior point of $H$ or whose peaks have distance from $\pa H$ diverging to infinity. From this point of view, we consider the variational problem
\[
S=\inf\Big\{\|\nabla u\|_{L^p(\R^n)}:\|u\|_{L^\pstar(\R^n)}=1\Big\}
\]
to be essentially equivalent to $\Phi(0)$.

\medskip

\noindent (b) The Escobar inequality boils down to the variational problem $\Phi(T)$ corresponding to $T=T_E$ for the constant
\begin{equation}
  \label{TEnp}
T_E=T_E(n,p)=\frac{\|U_E\|_{L^{\psharp}(\{x_1=1\})}}{\|U_E\|_{L^{\pstar}(\{x_1>1\})}}\,.
\end{equation}
Indeed, a simple scaling argument shows that, for every function $u(x)$ as in \eqref{escobar equality}, one has
\[
\frac{\|u\|_{L^{\psharp}(\pa H)}}{\|u\|_{L^{\pstar}(H)}}=T_E\,
\]
independently of the choices of $\lambda$ and, more surprisingly, of $z$. Thus, by the definition of $T_E$ and the characterization of equality cases in \eqref{escobar inequality lp}, we have
\[
\|u\|_{L^\psharp(\pa H)}=T_E\qquad\mbox{for every $u$ optimal function in \eqref{escobar inequality lp} with $\|u\|_{L^\pstar(H)}=1$}\,.
\]
As a consequence,
\[
\Phi(T_E)=E\,,
\]
and (the variational problem defined by) the Escobar inequality is equivalent to \eqref{phi p T} with $T=T_E$.

\subsection{What is known about $\Phi(T)$} As already noticed in Section \ref{overview}, a full characterization of $\Phi(T)$ in the important case $p=2$ was already given by Carlen and Loss in \cite{carlenloss}. The situation is quite different when $p\ne 2$. We now collect the information that, to the best of our knowledge, is all that is presently known about $\Phi(T)$. As we have just seen, $\Phi(0)=S$ by the Sobolev inequality, and we have a global linear lower bound
\begin{equation}
  \label{Phi escobar bound}
  \Phi(T)\ge E\,T\qquad\forall T\ge 0\,,
\end{equation}
with equality if $T=T_E$, thanks to the Escobar inequality. Another piece of information comes from the validity of the {\it gradient domain inequality} (see \cite[Section 7.2]{MVtwo2008} for the terminology adopted here) on $H$:
\begin{equation}
  \label{sobolev inequality lp gradient domain}
  \|\nabla u\|_{L^p(H)}\ge 2^{-1/n}\,S\,\|u\|_{L^\pstar(H)}\,,
\end{equation}
with equality if and only if there exists $\lambda>0$ such that
\[
u(x)=\lambda^{(n-p)/p}\,U_S(\lambda\,x)\qquad \forall x\in\R^n\,.
\]
The validity of \eqref{sobolev inequality lp gradient domain}, with equality cases, follows immediately by applying the Sobolev inequality \eqref{sobolev inequality lp} to the extension by reflection of $u$ to $\R^n$. The gradient domain inequality implies that
\begin{equation}
  \label{Phi constant bound}
  \Phi(T)\ge 2^{-1/n}\,S\,,\qquad\forall T\ge0
\end{equation}
with equality if and only if $T=T_0$ where
\[
T_0=\frac{\|U_S\|_{L^{\psharp}(\pa H)}}{\|U_S\|_{L^{\pstar}(H)}}\,.
\]
As we will prove later on (see Proposition \ref{prop: T and G}(i)),
\[
T_0<T_E\,,
\]
while clearly (by applying \eqref{sobolev inequality lp gradient domain} to an optimal function for \eqref{escobar inequality lp})
\begin{equation}
  \label{G0 less than GE}
  \Phi(T_0)=2^{-1/n}\,S<E=\Phi(T_E)\,.
\end{equation}
Next, we notice that, thanks to the divergence theorem and H\"older's inequality, for every non-negative
 $u$ that is admissible in $\Phi(T)$, we have
\[
\int_{\pa H}u^{\psharp}=\int_{\pa H}u^{\psharp}(-\ee)\cdot\nu_H=\psharp\int_{H}u^{\psharp-1}(-\nabla u)\cdot \ee< \psharp \|\nabla u\|_{L^p(H)}\,\|u\|_{L^\pstar(H)}^{\pstar/p'}
\]
where H\"older's inequality must be strict (otherwise, $u$ would just depend on $x_1$, and thus could not satisfy $u\in L^\pstar(H)$). As a consequence, we find that, with strict inequality,
\begin{eqnarray}
  \label{global lower bound}
  \Phi(T)>\frac{T^\psharp}\psharp\qquad\forall T>0\,.
\end{eqnarray}
Finally, given any open connected Lipschitz set $\Om\subset\R^n$, let us set
\[
\Phi_\Om(T)=\inf\Big\{\|\nabla u\|_{L^p(\Om)}:\|u\|_{L^\pstar(\Om)}=1\,,\|u\|_{L^\psharp(\pa\Om)}=T\Big\}\qquad T\ge0\,,
\]
(so that $\Phi_H=\Phi$ by \eqref{phi p T}), and define
\[
\ISO(\Om)=\frac{P(\Om)}{|\Om|^{(n-1)/n}}\,,
\]
where $P(\Om)$ and $|\Om|$ denote the perimeter (i.e., the $(n-1)$-dimensional measure of the boundary) and volume of $\Om$. With this notation, the {\it Euclidean isoperimetric inequality} takes the form
\begin{equation}
  \label{isoperimetric inequality}
  \ISO(\Om)\ge\ISO(B_1)\,,
\end{equation}
with equality if and only if $\Om=B_R(x)=\{y\in\R^n:|y-x|<R\}$ for some $x\in\R^n$ and $R>0$. The following {\it trace-Sobolev comparison theorem} was proved in \cite{MV2005}:
\begin{equation}
  \label{Phi maggivillani}
  \Phi_\Om(T)\ge\Phi_{B_1}(T)\,,\qquad\forall T\in\Big[0,\ISO(B_1)^{1/\psharp}\Big]\,,
\end{equation}
with the additional information that: (i) if $0<T\le \ISO(B_1)^{1/\psharp}$, $\Phi_\Om(T)=\Phi_{B_1}(T)$, and $\Phi_\Om(T)$ admits a minimizer, then $\Om$ is a ball; (ii) $\Phi_{B_1}$ is strictly concave (and decreasing) on $[0,\ISO(B_1)^{1/\psharp}]$. Notice that \eqref{Phi maggivillani} cannot hold on a larger interval of $T$s: indeed, $\Phi_{B_1}(T)=0$ forces $T=\ISO(B_1)^{1/\psharp}$, and so if $\Om$ is not a ball and thus $\ISO(\Om)>\ISO(B_1)$, then
\[
\Phi_{B_1}(\ISO(\Om)^{1/\psharp})>0=\Phi_\Om(\ISO(\Om)^{1/\psharp})\,.
\]
This said, we can apply \eqref{Phi maggivillani} with $\Om=H$ to obtain an additional lower bound on $\Phi$ on the interval $[0,\ISO(B_1)^{1/\psharp}]$.

The constant lower bound given in \eqref{Phi constant bound} is actually stronger than the other three lower bounds for some values of $T$. Indeed, there exists $\de>0$ such that
\begin{equation}
  \label{nontrivial}
  \Phi(T_0)>\max\Big\{1_{[0,\ISO(B_1)^{1/\psharp}]}(T)\Phi_{B_1}(T)\,,ET\,,\frac{T^\psharp}\psharp\,\Big\}\qquad\mbox{if $|T-T_0|<\de$}\,.
\end{equation}
By continuity, it suffices to check this assertion at $T=T_0$, and since \eqref{global lower bound} is strict for every $T>0$, we only need to worry about \eqref{Phi escobar bound} and \eqref{Phi maggivillani}. The fact that $\Phi(T_0)>\Phi_{B_1}(T_0)$ if $T_0\le\ISO(B_1)^{1/\psharp}$ follows by property (i) after \eqref{Phi maggivillani} and from the existence of a minimizer for $\Phi(T_0)$ shown in Theorem~\ref{thm: main} below. At the same time, $\Phi(T_0)>E\,T_0$, for otherwise, the explicit minimizer in $\Phi(T_0)$, that is the ``half-Sobolev optimizer'' $U_{S,0}$ (see \eqref{def of USt} below), would be optimal in \eqref{escobar inequality lp}, contradicting the characterization of equality cases for \eqref{escobar inequality lp} (which is already implicitly contained in \cite{Nazaret2006}, and is rigorously established in here). This proves \eqref{nontrivial}. We thus find the qualitative picture of the known lower bounds on $\Phi(T)$ depicted in
\begin{figure}
  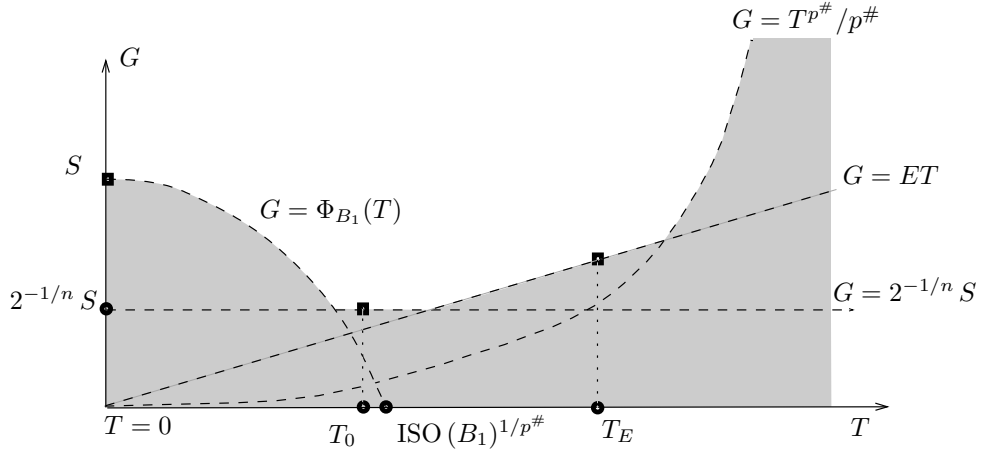
  \caption{{\small A qualitative picture of the known lower bounds on $\Phi(T)$. By combining the Sobolev and Escobar inequalities with the gradient domain inequality \eqref{sobolev inequality lp gradient domain}, the divergence theorem lower bound \eqref{global lower bound}, and the trace-Sobolev comparison theorem \eqref{Phi maggivillani} we conclude that $\Phi(T)$ lies above the gray region. The picture gives sharp information only for three values of $T$, namely $0$, $T_0$, and $T_E$, which are depicted by black squares. We also mention that numerical computations indicate the validity of $T_0<\ISO(B_1)^{1/\psharp}$ for every $n$ and $p$. We shall not give a formal proof of this fact, as it plays no role in our analysis.}}\label{fig situation}
\end{figure}
Figure \ref{fig situation}.

\subsection{Main results} Our main result consists of characterizing minimizers in $\Phi(T)$ for every $T>0$, and then using this knowledge to give a qualitative description of the behavior of $\Phi(T)$. The characterization result involves the following three families of functions:

\medskip

\noindent {\it Sobolev family:} Let $U_S$ be defined as in \eqref{eqn: US} and set, for every $t\in\R$,
\begin{equation}
  \label{def of USt}
  U_{S,t}(x)=\frac{U_S(x-t\,\ee)}{\|U_S(\textrm{id}-t\,\ee)\|_{L^{\pstar}(H)}} \qquad  x\in H\,,
\end{equation}
and
\begin{equation}\label{eqn: TS and GS}
T_S(t) = \| U_{S,t}\|_{L^\psharp(\pa H)}\,, \qquad
G_S(t) = \| \n U_{S,t}\|_{L^p( H)}\,.
\end{equation}
Thus, $U_{S,t}$ corresponds to translating the optimal function $U_S$ in the Sobolev inequality so that its maximum point lies at signed distance $t$ from $\pa H$, then multiplying the translated function by a constant factor to normalize the $L^\pstar$-norm in $H$ to be $1$.

\medskip

\noindent {\it Escobar family:} Letting $U_E$ be as in \eqref{eqn: UE}, we set for every $t<0$
\begin{equation}
  \label{def of UEt}
  U_{E,t}(x)=\frac{U_E(x-t\,\ee)}{\|U_E(\textrm{id}-t\,\ee)\|_{L^{\pstar}(H)}}\qquad x\in H\,.
\end{equation}
As noticed before, a simple computation (factoring out $|t|$ from $|x-t\,\ee|$ and then changing variables $y=-x/t$) shows that the trace and gradient norms of the $U_{E,t}$ are independent of $t<0$, and  we set
\begin{equation}\label{eqn: TE and GE}
\| U_{E,t}\|_{L^\psharp(\pa H)}=T_E\,,\qquad \| \n U_{E,t}\|_{L^p( H)}=G_E
\end{equation}
for these constant values. Each function $U_{E,t}$ is thus obtained by centering the fundamental solution of the $p$-Laplacian outside of $H$, and then by normalizing its $L^\pstar$-norm in $H$.

\medskip

\noindent {\it Beyond-Escobar family:} We consider the function
\begin{equation}
  \label{def of UBEt}
  U_{BE}(x)=(|x|^{p'}-1)^{(p-n)/p}\qquad |x|> 1\,,
\end{equation}
and define, for every $t<-1$,
\[
U_{BE,t}(x)=\frac{U_{BE}(x-t\,\ee)}{\|U_{BE}(\textrm{id}-t\,\ee)\|_{L^{\pstar}(H)}}\qquad x\in H\,.
\]
Correspondingly, for every $t<-1$, we set
\begin{equation}\label{eqn: TBE and GBE}
T_{BE}(t) = \| U_{BE,t}\|_{L^\psharp(\pa H)}\,, \qquad \qquad G_{BE}(t) =  \| \n U_{BE,t}\|_{L^p(H)}\,.
\end{equation}
As the name of this family of functions suggests, we later prove that $T_{BE}(t)>T_E$ for every $t<-1$, so that $\{U_{BE}(t)\}_{t<-1}$ enters the description of $\Phi(T)$ for $T>T_E$. Notice that \eqref{def of UBEt} defines a function on the complement of the unit sphere. The function $U_{BE,t}$ is thus obtained by centering this unit sphere {\it outside} of $H$, precisely at distance $|t|$ from $\pa H$, and the by normalizing its tail to have unit $L^\pstar$-norm in $H$.

\begin{theorem}[Characterization of minimizers of $\Phi(T)$]\label{thm: main}
  If $n\geq 2$ and $p\in (1,n)$, then for every $T>0$, there exists a minimizer in $\Phi(T)$ that is unique up to dilations and translations orthogonal to $\ee$. More precisely:
  \begin{enumerate}
  \item[(i)] the function $T_S(t)$ is strictly decreasing on $\R$ with range $(0,T_E)$ and with $T_S(0)=T_0<T_E$; in particular, for every $T\in(0,T_E)$, there exists a unique $t\in\R$ such that
      \begin{equation}
        \label{implicit S}
              T=T_S(t)\qquad \Phi(T)=G_S(t)
      \end{equation}
      and $U_{S,t}$ uniquely minimizes $\Phi(T)$ up to dilations and translations orthogonal to $\ee$;
  \item[(ii)] if $T=T_E$, then, up to dilations and translations orthogonal to $\ee$, $\{U_{E,t}:t<0\}$ is the unique family of minimizers of $\Phi(T_E)$;
  \item[(iii)] the function $T_{BE}(t)$ is strictly increasing on $(-\infty,-1)$ with range $(T_E,+\infty)$;  in particular, for every $T>T_E$ there exists a unique $t<-1$ such that
      \begin{equation}
        \label{implicit BE}
      T=T_{BE}(t)\qquad \Phi(T)=G_{BE}(t)
      \end{equation}
      and $U_{BE,t}$ uniquely minimizes $\Phi(T)$ up to dilations and translations orthogonal to $\ee$.
  \end{enumerate}
\end{theorem}

Theorem \ref{thm: main} provides an implicit description of $\Phi$ on $[0,\infty)$, and extends the Carlen--Loss theorem \cite{carlenloss} from the case $p=2$ to the full range $p\in(1,n)$. Notice that an implicit description of $\Phi_{B_1}$ on the interval $[0,\ISO(B_1)^{1/\psharp}]$ was obtained in \cite{MV2005}, and was at the basis of the further results obtained therein. (No characterization of $\Phi_{B_1}$ for $T>\ISO(B_1)^{1/\psharp}$ seems to be known.) Starting from the characterization of $\Phi$ obtained in Theorem \ref{thm: main}, we can obtain a quite complete picture of its properties, which is stated in the next result and illustrated in
\begin{figure}
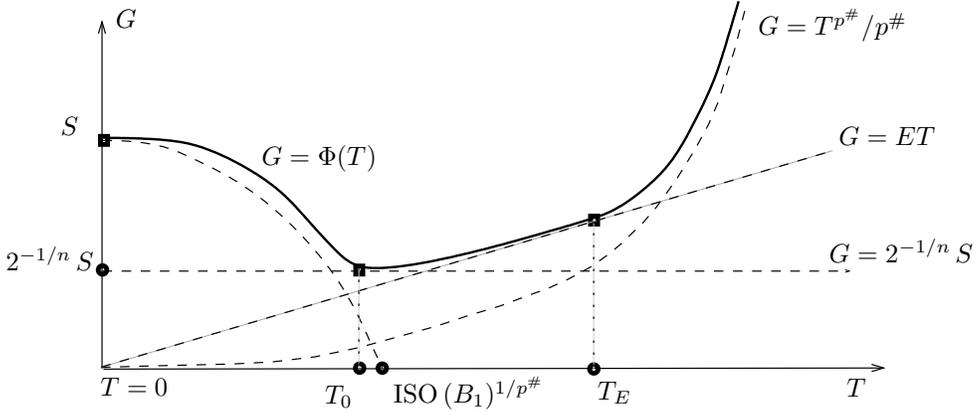\caption{{\small A qualitative picture of Theorem \ref{thm: properties of Phi}, which improves on the situation depicted in Figure \ref{fig situation}. First, since in Theorem \ref{thm: main} we have proved that $\Phi(T)$ always admits minimizers, we are sure that $\Phi(T)>\Phi_{B_1}(T)$ for every $T\in[0,\ISO(B_1)^{1/\psharp}]$, that is to say, the comparison theorem \eqref{Phi maggivillani} is never optimal (but at $T=0$). Notice also that the divergence theorem lower bound \eqref{global lower bound} turns out to be sharp, and is asymptotically saturated by the functions $U_{BE,t}$ as $t\to 1^-$.}}
  \label{fig summary}
\end{figure}
Figure \ref{fig summary}.

\begin{theorem}[Properties of $\Phi(T)$]\label{thm: properties of Phi}
If $n\ge 2$ and $p\in(1,n)$, then $\Phi(T)$ is differentiable on $(0,\infty)$, it is strictly decreasing on $(0,T_0)$ with $\Phi(0)=S$ and $\Phi(T_0)=2^{-1/n}\,S$ and strictly increasing on $(T_0,\infty)$ with
\begin{equation}\label{asymptotic behavior}
\Phi(T)=\frac{T^\psharp}{\psharp}+ o(1) \qquad \text{as } T\to \infty\,.
\end{equation}
Moreover, $\Phi(T)$ is strictly convex on $(T_0,+\infty)$, and there exists $T_*\in(0,T_0)$ such that $\Phi(T)$ is strictly concave on $(0,T_*)$.
\end{theorem}
We see from \eqref{asymptotic behavior} that the lower bound \eqref{global lower bound} is saturated asymptotically as $T \to \infty.$
A simple but interesting corollary of the characterization result obtained in Theorem \ref{thm: main} is the following comparison theorem, which is somehow complementary to \eqref{Phi maggivillani}, the main result in \cite{MV2005}.

\begin{corollary}[Half-spaces have the best Sobolev inequalities]\label{corollary bestbest} If $\Omega$ is a non-empty open set with Lipschitz boundary on $\R^n$, then
\[
\Phi_\Omega(T) \leq \Phi(T) \qquad\forall T\ge0\,.
\]
\end{corollary}

As it may be expected, the proof of Theorem \ref{thm: main} is based on a mass transportation argument in the spirit of \cite{cenv}. As we give more details on this point at the beginning of Section \ref{sec: OT}, we now comment on the meaning of these theorems in the geometrically relevant cases $p=2$ and $p=1$.

\subsection{The special case $p=2$} In this case, which implicitly requires $n\ge 3$, \eqref{phi p T} can be reformulated as a family of minimization problems {\it on conformally flat metrics on $H$},
\begin{equation}\label{PsiP}
  \Psi(P)=\inf\Big\{\int_H\,R_u\,d\vol_u+2\,(n-1)\int_{\pa H} h_u\,d\s_u:\vol_u(H)=1\,,P_u(H)=P\Big\}\qquad P\ge 0\,,
\end{equation}
which is related to the Yamabe problem on manifolds with boundary studied in the classical papers \cite{Escobar1988,Escobar92}. Here, we view $H$ as a conformally flat Riemannian manifold with boundary, endowed with the metric $u^{4/(n-2)}\,dx$. The volume and perimeter of a set $\Om\subset H$ with respect to this metric are computed as
\begin{equation}
  \label{volume and perimeter}
  \vol_u(\Om)=\int_\Om\,u^{2^\star}\,dx,\qquad P_u(\Om)=\int_{\pa\Om}\,u^{2^\sharp}\,d\H^{n-1}\,,
\end{equation}
while $R_u(x)$ and $h_u(x)$ stand, respectively, for the scalar curvature of $(H,u^{4/(n-2)}\,dx)$ at $x\in H$, and the mean curvature of $\pa H$ in $(H,u^{4/(n-2)}\,dx)$ at $x\in\pa H$ computed with respect to the outer unit normal $\nu_H$ to $H$. Explicitly,
\begin{equation}
  \label{scalar curvature and mean curvature}
  R_u=-\frac{4(n-1)}{n-2}\,\frac{\Delta u}{u^{(n+2)/(n-2)}}\,,\qquad
  h_u=-\frac{2}{n-2}\,\frac1{u^{n/(n-2)}}\,\frac{\pa u}{\pa x_1}\,.
\end{equation}
An integration by parts thus gives
\begin{eqnarray*}
  \int_{H}|\nabla u|^2
  &=&-\int_H\,u\,\Delta u-\int_{\pa H} u\,\frac{\pa u}{\pa x_1}
  \\
  &=&\frac{n-2}{4(n-1)}\int_H\,R_u\,d\vol_u+\frac{n-2}2\int_{\pa H} h_u\,d\s_u\,.
\end{eqnarray*}
In this way, we see the equivalence of the problems \eqref{phi p T} when $p=2$ and \eqref{PsiP} through the identities
\begin{equation*}
  \label{conformal interpretation}
  \Phi^{(2)}(T)=\Big(\frac{n-2}{4(n-1)}\Big)^{1/2}\,\Psi(T^{2^\sharp})^{1/2}
\qquad
\Psi(P)=\frac{4(n-1)}{n-2}\,\Phi^{(2)}(P^{1/2^\sharp})^2\,,
\end{equation*}
A standard argument shows that if $u$ is a positive minimizer for $\Phi(T)$ (with a generic $p\in(1,n)$), then there exist $\lambda,\s\in\R$ such that
\[
\begin{cases}
 -\Delta_p u = \lambda u^{\pstar-1} & \text{ in } \Ha\\
 -|\n u |^{p-2} \partial_{x_1} u = \sigma u^{\psharp - 1} & \text{ on } \pa \Ha \,.
 \end{cases}
\]
This basic fact, applied with $p=2$, implies that every minimizer in the variational problem \eqref{PsiP} is a conformally flat metric on $H$ with constant {\it scalar} curvature and with boundary of constant mean curvature. By \cite[Theorem 3.1]{carlenloss}, or with an alternative proof, by Theorem \ref{thm: main} with $p=2$, every minimizer actually has constant {\it sectional} curvature. Indeed, as a by-product of the characterization of minimizers of $\{\Phi(T)\}_{T\ge0}$, we deduce that, as $P$ increases from $0$ to $P_E=T_E^{2^\sharp}$, minimizing metrics in \eqref{PsiP} correspond to spherical caps of decreasing radii rescaled to unit volume. Their sectional curvature will be constant and positive along the way, while the constant mean curvature of the boundaries will initially be negative and then change sign in correspondence to hemispheres ($P=P_0=T_0^{2^\sharp}$). Then, as $P$ increases from $P_E$ to $+\infty$, minimizing metrics in \eqref{PsiP} correspond to suitable sections of the hyperbolic space, all with constant negative sectional curvature and constant positive mean curvature of the boundary. Thus, we have a transition from spherical to hyperbolic geometry along minimizing metrics in \eqref{PsiP}. These results are summarized in the following statement:

\begin{theorem}[Theorem 3.1 in \cite{carlenloss} or Theorem \ref{thm: main} with $p=2$]\label{thm: p=2}
For each $P >0$, a minimizing conformal metric $g_P$ exists in \eqref{PsiP} and is given,
uniquely up to dilations and translations orthogonal to $\ee$,  by
\begin{eqnarray*}
U_{S,t}^{4/(n-2)}\, dx & \quad \text{ for some }t \in \R &\qquad \text{ if }P \in (0, P_E)\,,\\
U_{E, t}^{4/(n-2)}\, dx & \quad \text{ for any }t <0& \qquad \text{ if }P =P_E\,,\\
U_{BE,t}^{4/(n-2)}\, dx &\quad  \text{ for some }t <-1 &\qquad \text{ if }P \in (P_E, \infty)\,.
\end{eqnarray*}

For $P \in (0, P_E)$, $(H, g_P)$ is isometric to a spherical cap $(\Sigma, g)$ with the standard metric induced by the embedding $S^n \hookrightarrow \R^{n+1}$ whose radius is determined by $P$; consequently, it has constant positive sectional curvature. The mean curvature of $\pa H$ is constant and negative for $0 <P<P_0 = T_0^{2^\sharp}$ and is constant and positive for $P_0 < P <P_E$.

For $P=P_E$, $(H, g_P)$ has zero sectional curvature and constant positive mean curvature of $\pa H$.

For $P \in (P_E, \infty)$,  $(H, g_P)$ has constant negative sectional curvature and is therefore a model for hyperbolic space. The mean curvature of $\pa H$ is constant and positive.
\end{theorem}

\subsection{The special case $p=1$} In this case, the minimization in \eqref{phi p T} takes place in the class of those $u\in L^1_{{\rm loc}}(H)$, vanishing at infinity, and whose distributional gradient $Du$ is a measure on $H$ with finite total variation, $|Du|(H)<\infty$. We thus consider the problems
\begin{eqnarray}
  \label{phi 1 T}
  \Phi(T)=\inf\Big\{|D u|(H):\|u\|_{L^{n/(n-1)}(H)}=1\,,\|u\|_{L^1(\pa H)}=T\Big\}\qquad T\ge0\,.
\end{eqnarray}
In the restricted class of characteristic functions $u=1_X$ for $X\subset H$, this is the relative isoperimetric problem in $H$ with an additional constraint (aside from the unit volume constraint) on the contact region between the boundary of $X$ and the boundary of $H$. In the notation of distributional perimeters, this restricted problem takes the form
\begin{equation}
  \label{phi 1 T sets}
  \Phi_{{\rm sets}}(T)=\inf\Big\{P(X;H):X\subset H\,, |X|=1\,,P(X;\pa H)=T\Big\}\qquad T\ge0\,,
\end{equation}
where $P(X;A)=\H^{n-1}(A\cap\pa X)$ whenever $X$ is an open set with Lipschitz boundary. The unique minimizers in \eqref{phi 1 T sets} are obtained by intersecting $H$ with balls (of suitable radius and centered at suitable distance from $\pa H$); see, e.g., \cite[Theorem 19.15]{maggiBOOK}, which also describes the relevance of \eqref{phi 1 T sets} in capillarity theory. In the original problem \eqref{phi 1 T}, one obtains scaled versions of the characteristic functions of these sets as minimizers; precisely, $u$ is a minimizer in \eqref{phi 1 T} if and only if $u(x)=\lambda^{n-1}\,1_X(\lambda\,x)$ for some $\lambda>0$ and $X$ a minimizer in \eqref{phi 1 T sets}. When $T=0$, \eqref{phi 1 T sets} is simply the Euclidean isoperimetric problem, and \eqref{phi 1 T} is the Sobolev inequality on functions of bounded variation. Notice that the Escobar inequality, in the case $p=1$, takes the simple form
\begin{equation}
  \label{escobar p1}
  |Du|(H)\ge\|u\|_{L^1(\pa H)}
\end{equation}
or, in more geometric terms, that is, for $u=1_X$ with $X\subset H$,
\[
P(X;H)\ge P(X;\pa H)\,.
\]
Along the lines of \eqref{global lower bound}, this follows by simply applying the divergence theorem on $X$ to the constant vector field $T(x)=\ee$ to get
\[
0=\int_X\Div(\ee)=\int_{H\cap\pa X}\nu_X\cdot \ee+\int_{\pa H\cap\pa X}(-\ee)\cdot \ee<P(X;H)-P(X;\pa H)
\]
where the inequality is strict as soon as $|X|>0$. The proof of \eqref{escobar p1} is analogous, and in particular, there is no equality case in \eqref{escobar p1} (i.e., a non-zero function realizing equality). In the case $p=1,$ Theorems~\ref{thm: main} and Theorem~\ref{thm: properties of Phi} take the following form.

\begin{theorem}\label{thm: p=1}
	For every $n\ge2$ and $T>0$ there exists a minimizer in \eqref{phi 1 T}, which is given, uniquely up to dilations and translations orthogonal to $\ee$, by
\[
U_{S,t}(x)=\frac{1_{B_1}(x-t\,\ee)}{\|1_{B_1}(\cdot-t\,\ee)\|_{L^{n'}(H)}}\qquad x\in H
\]
for some $t\in(-1,1)$.
 The function $\Phi(T)$ defined by \eqref{phi 1 T} is a smooth function of $T>0$ given by the parametric curve
\[
	\Phi(T_S(t))  =G_S(t)  \qquad -1<t<-1\,,
\]
where $T_S(t)=\|U_{S,t}\|_{L^1(\pa H)}$ and $G_S(t)=|DU_{S,t}|(H)$. If we set $T_0=T_S(0)$, then $\Phi(T)$ is strictly decreasing on $(0,T_0)$ and strictly increasing on $(T_0,\infty)$, with $\Phi(0)=\ISO(B_1)$ and $\Phi(T_0) =2^{-1/n}\ISO(B_1)$. Moreover, $\Phi$ is strictly convex on $(T_0, \infty)$, there exists $T_*\in(0,T_0)$ such that $\Phi(T)$ is strictly concave on $(0, T_*)$,  and $\Phi(T) = T + o(1)$ as $T \to \infty$.
\end{theorem}

We note that in the case $p=1$, we have a single minimizing family, corresponding to the Sobolev family of the case $p\in(1,n)$, but no Escobar or beyond-Escobar families. This is a reflection of the fact that
\[
\lim_{p\to 1^+} T_E(n,p) = \infty\,,
\]
proven in Proposition \ref{prop TEnp} below.
This fact indicates that no analogues of the Escobar or beyond-Escobar families exist for $p=1$. In the same vein, one notices that the $\Phi$ curve asymptotically has the same slope (equal to $1$) as the (limit position as $p\to 1^+$ of the) Escobar line.

\subsection{Organization of the paper}
In Section~\ref{sec: OT}, we use a mass transportation argument to prove a family of inequalities which will serve as a key tool for proving the main results. In Section~\ref{sec: Phi}, we prove Theorems~\ref{thm: main}, \ref{thm: properties of Phi}, and \ref{thm: p=1}. Finally, in Appendix~\ref{sec: appendix}, we address some technical points related to the mass transportation argument.

\medskip

\noindent{\bf Acknowledgments.} We thank Eric Carlen and Michael Loss for their advice concerning \cite{carlenloss}. RN supported by the NSF Graduate Research Fellowship under Grant DGE-1110007. FM supported by the NSF Grants DMS-1265910 and DMS-1361122.

\section{Mass transportation argument}\label{sec: OT}
The starting point of our analysis is the mass transportation proof of the Sobolev inequality from \cite{cenv}. This argument, whose origin can be traced back to \cite{knothe,Gromov}, was exploited in \cite{MV2005} to prove
a parameterized ``mother family'' of trace Sobolev inequalities on arbitrary Lipschitz domains, leading to the sharp comparison theorem stated in \eqref{Phi maggivillani}. In \cite{Nazaret2006}, this method of proof is adapted to obtain the Escobar inequality for every $p\in(1,n)$. It is important to mention that, as already shown in \cite{cenv} (see also \cite{agk,MVtwo2008,Nguyen15}), this optimal transportation argument can also be applied to a very interesting special family of Gagliardo--Nirenberg inequalities, having some Faber-Krahn and log-Sobolev inequalities as limit cases.

At the core of this paper is a new iteration of this by-now-classical mass transportation argument. This iteration lies in between the ones of \cite{MV2005} and \cite{Nazaret2006}. In Theorem \ref{thm: mother} we implement the same trick introduced in \cite{Nazaret2006}, namely subtracting a unit vector from the Brenier map, but with the seemingly harmless addition of an intensity parameter $t$. (To be precise, the argument in \cite{Nazaret2006} corresponds to the choice $t=-1$ in the proof of Theorem \ref{thm: mother}.) This simple expedient leads to obtain a new parameterized ``mother family'' of trace-Sobolev inequalities on the half-space, whose equality cases (see Theorem \ref{thm: equality cases} below) are given by the functions $U_{S,t}$, $U_{E,t}$ and $U_{BE,t}$ introduced in \eqref{def of USt}, \eqref{def of UEt} and \eqref{def of UBEt}. This means that each inequality in the mother family provides a sharp trace-Sobolev bound, which thus agrees with $\Phi(T)$ for a specific value of $T$ depending on $t$. By adopting the same point of view of \cite{MV2005}, where the $\Phi$-function of the ball was computed for a special range of $T$, in Section \ref{sec: Phi} we exploit this implicit description of $\Phi(T)$ in order to prove Theorem \ref{thm: main}.

Let us now recall some facts from the theory of optimal transportation. Given a (Borel regular) probability measure $\mu $ on $\R^n$ and a Borel measurable map $T:\R^n\to \R^n$, the {\it push-forward of $\mu$ through $T$} is the probability measure defined by
\[
T \# \mu (A) = \mu (T^{-1} (A)) \qquad \forall A \subset \R^n.
\]
As a consequence of this definition, for every Borel measurable function $\xi:\R^n\to[0,\infty]$ we have
\begin{equation}\label{eqn: integral push forward}
\int_{\R^n} \xi\, d T\#\mu = \int_{\R^n} \xi\circ T \, d \mu \,.
\end{equation}
If $F \, dx$ and $G \, dx$ are absolutely continuous probability measures on $\R^n$, then the Brenier-McCann theorem (see \cite{Brenier91, McCann97} or \cite[Cor. 2.30]{VillaniTopics}) ensures the existence of a lower semicontinuous convex function $\vphi: \R^n \to \R\cup \{ +\infty\} $ such that
\begin{equation}\label{eqn: Brenier pf}
(\n \vphi ) \# F\,dx = G\, dx \,.
\end{equation}
By convexity, $\vphi$ is differentiable a.e. on the open convex set $\Om$ defined as the interior of $\{\vphi<\infty\}$, its gradient satisfies
\[
\n\vphi\in (BV\cap L^\infty)_{{\rm loc}}(\Om;\R^n)\,,
\]
and $F\,dx$ is concentrated on $\Om$ with
\begin{equation}\label{eqn: supports}
	\supp(G\, dx) = \overline{\n \vphi (\supp(F\, dx))}\,,
\end{equation}
thanks to \eqref{eqn: Brenier pf}. The map $T=\n\vphi$ is called the {\it Brenier map} between $F\,dx$ and $G\,dx$, and, as shown in \cite{McCann97} (cf. \cite[Theorem 4.8]{VillaniTopics}), it satisfies the Monge-Ampere equation
\begin{equation}\label{MA}
F(x) =  G(\n \vphi (x) ) \, \dett \nabla^2 \vphi(x)\qquad\mbox{ a.e. on $\supp(F\,dx)$}\,.
\end{equation}
Notice that the distributional gradient $DT$ of $T$ is an $n\times n$-symmetric tensor valued Radon measure on $\Om$. In \eqref{MA} we have set $\nabla^2\vphi=\nabla T$ where $DT=\nabla T\,dx+D^sT$ is the decomposition of $DT$ with respect to the Lebesgue measure on $\Om$. Notice that $\nabla T\,dx\le DT$ on $\Om$, and thus, setting $\Div T=\trace(\nabla T)$ and denoting by $\DIV T$ the distributional divergence of $T$, we have
\[
\Div\,T\,dx\le \DIV\,T\qquad\mbox{as measures on $\Om$}\,.
\]
Since $\nabla T(x)$ is positive semidefinite, by the arithmetic-geometric mean inequality
\[
(\dett \nabla^2 \vphi(x))^{1/n}=(\det\nabla T(x))^{1/n}\le\frac{\Div T(x)}n\qquad\mbox{for a.e. $x\in\Om$}\,,
\]
we finally conclude that
\begin{equation}
  \label{am inq x}
  (\dett \nabla^2 \vphi)^{1/n}\,dx\le \frac{\DIV T}n\qquad\mbox{as measures on $\Om$}\,.
\end{equation}

\begin{theorem}\label{thm: mother}
  If $n \geq 2$, $p\in [1,n)$, and $f$ and $g$ are non-negative functions in $L^1_{{\rm loc}}(H)$, vanishing at infinity, with
  \begin{equation}
    \label{hp on f and g}
    \left\{\begin{split}
      &\mbox{$\int_H|\nabla f|^p<\infty$ and $\int_{H}|x|^{p'}g^\pstar<\infty$ if $p>1$}
      \\
      &\mbox{$|Df|(H)<\infty$ and $\spt\,g\cc\ov{H}$ if $p=1$}
      \\
      &\| f\|_\norm = \|g\|_\norm=1
    \end{split}
    \right .
  \end{equation}
  then for every $t\in\R$, we have
  \begin{equation}\label{mother}
  n \int_{\Ha} g^{\psharp} \, dx \leq \psharp
  \|\n f\|_\pnorm Y(t,g) +t  \int_{\pa \Ha } f^{\psharp}\, d \Hi^{n-1}
  \end{equation}
  where we let
  \begin{equation} \label{eqn: Y def}
  Y(t,g) = \begin{cases}
  \Big( \int_{\Ha} g^{\pstar} | x  -t\,\ee|^{p'} \, dx\Big)^{1/p'} & \text{ if } p>1\,,\\
  \sup\{ |x-t\,\ee| : x \in \supp(g) \}& \text{ if } p=1\,,
  \end{cases}
  \end{equation}
  and where $\|\n f\|_\pnorm$ is replaced by $|Df|(H)$ when $p=1.$
\end{theorem}

\begin{remark}
  {\rm Let us first recall that the assumption that $f$ is vanishing at infinity means that $|\{f>t\}|<\infty$ for every $t>0$. Next we notice that, by \eqref{escobar inequality lp}, \eqref{hp on f and g} implies $f\in L^\psharp(\pa H)$, so that the multiplication by a possibly negative $t$ on the right-hand side of \eqref{mother} is of no concern. Finally, we notice that \eqref{mother} implies that $g\in L^\psharp(H)$, but this fact can be more directly deduced by means of H\"older's inequality from the assumptions on $g$ stated in \eqref{hp on f and g}.}
\end{remark}

\begin{proof}
Arguing by approximation, it suffices to prove \eqref{mother} when $f\in C^1_c(\overline{H})$ (that is, $f$ admits an extension in $C^1_c(\R^n)$). Let us set $F=1_H\,f^\pstar$ and $G=1_H\, g^\pstar$ and consider the Brenier map $\n\vphi$ between the probability measures $F\,dx$ and $G\,dx$. In this way,
$T=\n\vphi\in (BV\cap L^\infty)_{\rm loc}(\Om;\R^n)$ with $\Om$ defined as above and $F\,dx$ is concentrated on $\Om$. By \eqref{eqn: Brenier pf}, \eqref{eqn: integral push forward} (applied with $\xi=1_{\{G>0\}}\,G^{-1/n}$), \eqref{MA} and \eqref{am inq x} respectively, we have
\begin{equation}\label{eqn: AMGM inequality}
\int_H\,g^\psharp=\int_{\R^n}  G^{1-1/n}  =\int_{\R^n} G(\n \vphi)^{-1/n} F  =\int_{\R^n} (\dett \nabla^2 \vphi)^{1/n} F^{1-1/n}  \le
\frac{1}{n} \int_{\R^n}  F^{1-1/n} \,d(\Div T)\,.
\end{equation}
By a slight modification of \cite{Nazaret2006},  we subtract the divergence-free vector field $t\,\ee$ from $T$,
\[
\int_{\R^n}  F^{1-1/n} \,d(\DIV T)= \int_{H} f^\psharp  \,d(\DIV S) \,,\qquad S=T-t\,\ee\,,
\]
where $S\in (BV\cap L^\infty)_{{\rm loc}}(\Om;\R^n)$. By the trace theorem for $BV$ functions (see e.g. \cite[Theorem 1, p.177]{EvansGariepyBOOK}), $S$ has a trace $S\in L^1_{{\rm loc}}(\Om\cap\pa H)$ such that
\[
\int_{H}\,\psi\,d(\DIV S)=-\int_H\,\nabla\psi\cdot S-\int_{\pa H}\,\psi\,(S\cdot\ee)\,,\qquad\forall\psi\in C^1_c(\Om\cap\ov{H})\,.
\]
We now use the assumption that $f\in C^1_c(\ov{H})$ (together with the fact that $F\,dx$ is concentrated on $\Om$) to apply  this identity with $\psi=f^\psharp$. In this way, we find
\begin{equation*}\label{eqn: formal IBP}
 \int_H f^\psharp\,d(\Div S)= - \psharp \int_{\Ha} f^{\psharp -1} \n f \cdot S\, dx  - \int_{\pa {\Ha}} f^{\psharp} S\cdot\ee d\Hi^{n-1}\,.
 \end{equation*}
Since $\ov{T(\spt( F\,dx))}=\spt(G\,dx)\subset\ov{H}$, by standard properties of the trace operator we have $S(x) \cdot(-\ee) \leq t$ for $\H^{n-1}$-a.e. on $x\in\supp(f)\cap\pa H$. Thus, in summary,
 \begin{equation}\label{eqn: formal IBP new }
 n\,\int_H\,g^\psharp\leq - \psharp \int_{\Ha} f^{\psharp -1} \n f \cdot( T  - t\,\ee)   + t\int_{\pa {\Ha}} f^{\psharp} \, d\Hi^{n-1}\,.
 \end{equation}
Finally, we bound the first term on the right hand side of \eqref{eqn: formal IBP new }. In the case that $p\in(1,n)$, by using H\"{o}lder's inequality and the transport condition \eqref{eqn: integral push forward} we find
\begin{align}\nonumber
-\psharp \int_{\Ha} f^{\psharp -1} \n f \cdot( T  - t\,\ee)
&  \leq
\psharp \|\n f\|_\pnorm \Big( \int_{\Ha} f^{\pstar} | T(x)  -t\,\ee|^{p'} \, dx\Big)^{1/p'} \\
 & \label{eqn: Holder step}
  =
\psharp \|\n f\|_\pnorm \Big( \int_{\Ha} g^{\pstar} | x  - t\,\ee|^{p'} \, dx\Big)^{1/p'}.
\end{align}
Combining this with \eqref{eqn: formal IBP new } implies \eqref{mother}. In the case $p=1$, in place of H\"older's inequality, we simply use \eqref{eqn: supports} and the fact that $\psharp=1$ to bound the left-hand side of \eqref{eqn: Holder step} by $Y(t,g)\,|Df|(H)$.
\end{proof}

In order to analyze the mother family of inequalities of Theorem \ref{thm: mother} we will need a characterization of the corresponding equality cases, which involves the functions $U_{S,t}$, $U_{E,t}$ and $U_{BE,t}$ previously introduced in \eqref{def of USt}, \eqref{def of UEt} and \eqref{def of UBEt}. Following \cite{cenv}, given two non-negative measurable functions $f$ and $g$, we call $f$ a {\it dilation-translation image} of $g$ if there exist $C > 0, \lambda \neq 0,$ and $x_0\in \R^n$ such that $f(x) = C g(\lambda(x-x_0))$. Since \eqref{mother} is not invariant with respect to translations in the $\ee$ direction, we distinguish that $f$ is a {\it dilation-translation image} of $g$ {\it orthogonal to $\ee$} if $f$ is a dilation-translation image of $g$ with $x_0 \cdot \ee = 0.$ If $\int_{\Ha} f^\pstar \, dx = \int_{\Ha} g^\pstar\, dx$ and  $f$ is a dilation-translation image of $g$ orthogonal to $\ee$, then $C$ must be equal to $\lambda^{(n-p)/n}$, and the Brenier map pushing forward $f^\pstar \, dx$ onto $g^\pstar\,dx$ satisfies $\n \vphi = \lambda ({\rm Id} - x_0)$ with $x_0 \cdot \ee = 0$. With this terminology at hand, we state the required characterization theorem:

\begin{theorem}\label{thm: equality cases} Under the same assumptions of Theorem \ref{thm: mother}, suppose that
  \begin{equation}\label{mother equality case}
  n \int_{\Ha} g^{\psharp} \, dx =\psharp
  \|\n f\|_\pnorm Y(t,g) +t  \int_{\pa \Ha } f^{\psharp}\, d \Hi^{n-1}\,,\qquad \int_{\pa H}f^\psharp>0\,,
  \end{equation}
  where $|Df|(H)$ replaces  $\|\n f\|_\pnorm$ when $p=1$.

  \medskip

  \noindent {\it If $p\in(1,n)$}, then \eqref{mother equality case} holds for $t\ge 0$ if and only if $f$ and $g$ are both dilation-translation images orthogonal to $\ee$ of $ U_{S,t}$; and for $t<0$ if and only if $f$ and $g$ are both dilation-translation images orthogonal to $\ee$ of either $U_{S,t}$,  $U_{E,t}$, or $U_{BE,t}$.

  \medskip

  \noindent {\it If $p=1$}, then \eqref{mother equality case} can hold only for $t\in(-1,1)$. For such $t$, \eqref{mother equality case} holds if and only if $f$ and $g$ are dilation-translation images orthogonal to $\ee$ of  $ U_{S,t}$.
%
%
\end{theorem}

Since the proof of Theorem \ref{thm: equality cases} is just a technical variant of a similar argument from \cite{cenv}, we postpone its discussion  to the appendix.

\section{Study of $\Phi$} \label{sec: Phi}
\subsection{The case $p \in (1,n)$} By Theorem \ref{thm: equality cases}, if equality is achieved in the mother inequality \eqref{mother} by a triple $(t,f,g)$ with $\int_{\pa H}f^\psharp> 0$, then we either have $f=g=U_{S,t}$ when $t\ge0$, or
\[
f=g=U_{S,t}\qquad\mbox{or} \qquad f=g=U_{E,t} \qquad {\rm or }\qquad f=g=U_{BE,t}
\]
when $t<0$ (with the third possibility only when $t<-1$). The same scaling argument used to show that $G_E(t)=G_E$ and $T_E(t)=T_E$ for every $t<0$ (see \eqref{eqn: TE and GE}) serves to check that $Y(t, U_{E,t}) = |t|\, Y_E$, where we let $Y_E= Y(-1, U_{E,-1})$ and $Y(t, g)$ be as defined in \eqref{eqn: Y def}. Therefore, recalling the notation introduced in \eqref{eqn: TE and GE}, equality in \eqref{mother} for the Escobar family implies that
\begin{equation}\label{eqn: eq E}
n \int_{\Ha} U_{E,t}^\psharp \, dx = -t\psharp G_E Y_E + t\, T_E^\psharp \qquad \forall t<0\,.
\end{equation}
Similarly, let us
define the functions
\[
Y_S(t)= Y(t,U_{S,t}) \qquad \text{ and }\qquad Y_{BE}(t) = Y(t, U_{BE,t})\,.
\]
Then, recalling the definitions in \eqref{eqn: TS and GS} and \eqref{eqn: TBE and GBE}, equality in \eqref{mother} for the Sobolev and beyond-Escobar families implies the identities
\begin{equation}\label{eqn: eq other}
\begin{split}
n \int_{\Ha} U_{S,t}^\psharp \, dx &= \psharp G_S(t) Y_S(t) + t\,T_S(t)^\psharp
 \qquad \qquad \forall t\in \R \,, \\
n \int_{\Ha} U_{BE,t}^\psharp \, dx &= \psharp G_{BE}(t) Y_{BE}(t) + t\,T_{BE}(t)^\psharp \qquad \forall t<-1 \,.
\end{split}
\end{equation}
From \eqref{eqn: eq E} and \eqref{eqn: eq other}, Theorems~\ref{thm: mother} and \ref{thm: equality cases} yield the following corollary.

\begin{corollary}\label{cor: cor of mother}
If $h\in L^1_{{\rm loc}}(H)$ is a non-negative function vanishing at infinity with $\nabla h\in L^p(H;\R^n)$ and $ \| h \|_{L^{\pstar}(H)}=1$, then, \begin{align}\label{eqn: consequence of mother}
\psharp Y_S(t) G_S(t)  + t\, T_S(t)^{\psharp}& \leq \psharp Y_S(t)\|\n h\|_\pnorm + t \, \| h\|_\tnorm^\psharp \qquad \ \ \forall t\in \R\,, \\
\label{eqn: consequence of mother BE}
\psharp Y_{BE}(t) G_{BE}(t)  + t\, T_{BE}(t)^{\psharp}& \leq \psharp Y_{BE}(t)\|\n h\|_\pnorm + t \, \| h\|_\tnorm^\psharp \qquad \forall t<-1\,, \\
\label{eqn: consequence of mother E}
 \psharp Y_E G_E - \, T_E^{\psharp} &\leq  \psharp Y_E\|\n h\|_\pnorm - \, \| h\|_\tnorm^\psharp.
\end{align}
Furthermore, equality in \eqref{eqn: consequence of mother} (resp. \eqref{eqn: consequence of mother BE}, \eqref{eqn: consequence of mother E}) is attained if and only if $h$ is a dilation-translation image orthogonal to $\ee$ of $U_{S,t}$ (resp. $U_{BE,t}$, $U_{E,t}$).
Particularly,
\begin{eqnarray*}
\| h\|_\tnorm = T_S(t) & \implies & G_S(t)\leq \|\n h\|_\pnorm\,; \\
\| h\|_\tnorm = T_{BE}(t) & \implies & G_{BE}(t)\leq \|\n h\|_\pnorm\,; \\
\| h\|_\tnorm = T_E & \implies & G_E\leq \|\n h\|_\pnorm\,,
\end{eqnarray*}
and the following identities hold
\begin{equation}\label{eqn: param}
\Phi(T_S(t) ) = G_S(t)\quad\forall t\in\R\,, \qquad \Phi(T_{BE}(t) ) = G_{BE}(t)\quad\forall t<0\,,\qquad \Phi(T_E) = G_E\,.
 \end{equation}
\end{corollary}

Next, we prove some properties of the Sobolev and beyond-Escobar families.

\begin{proposition} \label{prop: T and G} The following properties hold:
\begin{enumerate}
  \item[(i)] $T_S$ is strictly decreasing on $\R$ with range $(0, T_E)$, and $T_S(0)=T_0<T_E$;
  \item[(ii)] $G_S$ is strictly increasing on $[0,\infty)$ with range $[2^{-1/n}S, G_E)$, and is strictly decreasing on $(-\infty,0)$ with range $(2^{-1/n}S,G_E)$;
  \item[(iii)] $T_{BE}(t)$ is strictly increasing for $t<-1$ with range $(T_E, \infty)$;
  \item[(iv)] $G_{BE}(t)$ is strictly increasing for $t<-1$ with range $(G_E, \infty)$.
\end{enumerate}
\end{proposition}

\begin{proof} {\it Step 1: Monotonicity of $T_S(t)$ and $T_{BE}(t)$.}
 Fix $t_1, t_2 \in \R$ and suppose $T_S(t_1) = T_S(t_2) = T.$ Then, \eqref{eqn: consequence of mother} implies that
\begin{eqnarray*}\label{eqn: one}
\psharp Y_S(t_1) G_S(t_1) + t_1 T^\psharp \leq \psharp Y_S(t_1) G_S(t_2) + t_1 T^\psharp , &  \rm{ thus } & G_S(t_1) \leq G_S(t_2),\ \ \rm{ and} \\
\label{eqn: two*}
 \psharp Y_S(t_2) G_S(t_2) + t_2 T^\psharp  \leq \psharp Y_S(t_2) G_S(t_1) + t_2 T^\psharp , &  \rm{ thus }  & G_S(t_2) \leq G_S(t_1).
\end{eqnarray*}
That is, $G_S(t_1)= G_S(t_2) = G$. Hence, $U_{S,t_2}$ attains equality in \eqref{eqn: consequence of mother} with $t=t_1.$ Uniqueness in \eqref{eqn: consequence of mother} then implies that $t_1 = t_2$. We conclude that $T_S(t)$ is injective, and, as $T_S(t)$ is continuous, it is strictly monotone for $t \in \R$. The identical argument using \eqref{eqn: consequence of mother BE} shows that $T_{BE}$ is strictly monotone for all $t<-1$.
\medskip

\noindent{\it Step 2: Piecewise monotonicity of $G_S(t)$ and $G_{BE}(t)$.}
Fix $t_1, t_2\geq 0$ and suppose that $G_S(t_1) = G_S(t_2) = G$.
Then, \eqref{eqn: consequence of mother} implies that
\begin{eqnarray*}
\psharp Y_S(t_1) G + t_1 T_S(t_1)^\psharp  \leq \psharp Y_S(t_1) G + t_1 T_S(t_2)^\psharp , &  \rm{ thus } & T_S(t_1) \leq T_S(t_2), \ \ \rm{ and}\\
\psharp Y_S(t_2) G + t_2 T_S(t_2)^\psharp  \leq \psharp Y_S(t_2) G + t_2 T_S(t_1)^\psharp, &  \rm{ thus }& T_S(t_2) \leq T_S(t_1).
\end{eqnarray*}
Since $T_S(t)$ is injective, we conclude that $t_1 = t_2$. Thus, $G_S(t)$ is strictly monotone for $t\ge0$. The analogous argument
shows that $G_S(t)$ is strictly monotone for $t< 0$
and that $G_{BE}(t)$ is strictly monotone for $t<-1.$
\medskip

\noindent {\it Step 3: Limit values of $T_{S}(t)$ and $G_S(t)$.} As $U_{S,t}$ is a renormalized translation of the optimal function $U_S$ in \eqref{sobolev inequality lp}, centered at $t\, \ee$, it is clear that $T_S(t) \to 0$ and $G_S(t) \to S$  as $t\to \infty$.
To compute the limit as $t\to -\infty$, let us set
\[
\gamma_t(x)= (1+ |x-t\,\ee|^{p'})^{-1}=|t|^{-p'}\,(|t|^{-p'}+|y+\ee|^{p'})^{-1}
\]
for $t<0$ and $y=-x/t$.
With this notation,
\begin{equation}\label{eqn: T and G}
T_S(t) =\frac{ \big(
	 \int_{\pa {\Ha}} \g_t^{n-1}\, d\H^{n-1}  \big)^{1/\psharp}}
	  {	\big( \int_{{\Ha}} \g_t^n\, dx	 \big)^{1/\pstar}},
\qquad
G_S(t) = \frac{	(n-p)\big(\int_{\Ha} \g_t^n |x-t\,\ee|^{p'} \, dx		\big)^{1/p} }{ (p-1)\big( \int_{{\Ha}} \g_t^n\, dx 		\big)^{1/\pstar}}.
\end{equation}
Now, suppose $t<0$ and let $\sigma = -(n-p)/(p-1).$ After factoring out $-t$ and changing variables, we find that
\begin{align*}
 \int_{\pa {\Ha}} \g_t^{n-1}\, d\H^{n-1} &=
|t|^{-p'(n-1)+(n-1)} \int_{\pa {\Ha}}  (|t|^{-p'} + |y + \ee |^{p'})^{-(n-1)} \, d\H^{n-1}_y\,,\\
\ \int_{\Ha}  \g_t^{n} \, dx & =  |t|^{-p'n+n}  \int_{\Ha} (|t|^{-p'} + |y+\ee|^{p'} )^{-n}\, dy\,,\\
\int_{\Ha} \g_t^{n}|x-t\,\ee|^{p'}\, dx & =  |t|^{-p'n+p'+n} \int_{\Ha} (|t|^{-p'} + |y+\ee|^{p'} )^{-n}|y+\ee|^{p'} \, dy\,.
\end{align*}
Since
\[
\frac{-p'(n-1)+(n-1)}{\psharp}-\frac{-p'n+n}{\pstar}=0\qquad\mbox\qquad \frac{-p'n+p'+n}{p}+\frac{p'n-n}{\pstar}=0\,,
\]
we find that, setting
\[
\bar\g_t(y)=(|t|^{-p'}+|y+\ee|^{p'})^{-1}\qquad y\in H\,,
\]
we have
\[
T_S(t)=\frac{\Big(\int_{\pa H}\bar\g_t^{n-1}\Big)^{1/\psharp}}{\Big(\int_{H}\bar\g_t^{n}\Big)^{1/\pstar}}
\qquad
G_S(t) = \frac{	(n-p)\big(\int_{\Ha} \bar\g_t^n |y+\ee|^{p'} \,dy\big)^{1/p} }{ (p-1)\big( \int_{{\Ha}} \bar\g_t^n		\big)^{1/\pstar}}
\]
By monotone convergence, we thus find that
\[
\lim_{t\to-\infty}T_S(t)=\frac{\|U_E(\cdot+\ee)\|_{L^\psharp(\pa H)}}{\|U_E(\cdot+\ee)\|_{L^\pstar(H)}}=T_E
\qquad
\lim_{t\to -\infty}G_S(t) = \frac{\|\nabla U_E(\cdot+\ee)\|_{L^p(H)}}{\|U_E(\cdot+\ee)\|_{L^\pstar(H)}}=G_E\,,
\]
as claimed. Having shown that $T_S$ is smooth and injective on $\R$ with $T_S(+\infty)=0$ and $T_S(-\infty)=T_E>0$, we deduce that $T_S$ is strictly decreasing on $\R$ with range $(0,T_E)$. Since $T_0=T_S(0)<T_S(-\infty)=T_E$, we have completed the proof of statement (i). Similarly, the first part of (ii) follows since $G_S(0)=2^{-1/n}S<S=G_S(+\infty)$ and $G_S$ is smooth and injective on $[0,\infty)$. Similarly, the injectivity of $G_S$ on $(-\infty,0)$ together with the fact that by \eqref{sobolev inequality lp gradient domain} (recall \eqref{G0 less than GE}) $G_S(0)=2^{-1/n}\,S<E=G_E=G_S(-\infty)$ implies that $G_S$ is strictly decreasing on $(-\infty,0)$ with range $(2^{-1/n}S,E)$. This proves statement (ii).

\medskip

\noindent {\it Step 4: Limit values of $T_{BE}(t)$ and $G_{BE}(t)$.} With an argument identical to that given for $T_S$ and $G_S,$ we establish that $T_{BE}(t) \to T_E$ and $G_{BE}(t) \to G_E$ as $t \to -\infty$. To compute the limit as $t\to -1^+$, we first notice that, for every $t<-1$ and setting $\e=|t|-1$,
\[
\int_{\pa H}U_{BE}(x-t\,\ee)^\psharp\ge\int_{B_{|t|+\e}(t\ee)\cap\pa H}\frac{d\H^{n-1}}{(|x-t\,\ee|^{p'}-1)^{n-1}}\,.
\]
Since $B_{|t|+\e}(t\ee)\cap\pa H$ is a $(n-1)$-dimensional disk of radius $\sqrt{(|t|+\e)^2-t^2}=\sqrt{2\e\,|t|+\e^2}\ge c\,\sqrt{\e}$, and since
$|x-t\,\ee|^{p'}-1\le (|t|+\e)^{p'}-1\le C\,\e$ for constants $c$ and $C$ depending on $n$ and $p$ only, we find that
\begin{equation}
  \label{ts1}
\int_{\pa H}U_{BE}(x-t\,\ee)^\psharp\ge \frac{c}{\e^{(n-1)/2}}=\frac{c}{|t+1|^{(n-1)/2}}\,.
\end{equation}
At the same time, we have
\[
\int_H U_{BE}(x-t\,\ee)^\pstar=\int_H(|x-t\,\ee|^{p'}-1)^{-n}\,dx=\int_{-t}^\infty\,(r^{p'}-1)^{-n}\,\H^{n-1}\big(H\cap\pa B_r(-t\,\ee)\big)\,dr
\]
where, thanks to the coarea formula,
\[
\H^{n-1}\big(H\cap\pa B_r(-t\,\ee)\big)=c(n)\,r^{n-1}\int_{-t/r}^1\,(1-s^2)^{(n-3)/2}\,ds\,.
\]
Since $1\le (1+s)^{(n-3)/2}\le C(n)$ for $s\in(-t/r,1)$ and
\[
r^{n-1}\,\int_{-t/r}^1\,(1-s)^{(n-3)/2}\,ds
=C(n)\,r^{n-1}\,(1+t/r)^{(n-1)/2}=C\,r^{(n-1)/2}\,(r+t)^{(n-1)/2}\,,
\]
we conclude that
\begin{equation}
  \label{lbbb}
  c(n)\,\le \frac{\H^{n-1}\big(H\cap\pa B_r(-t\,\ee)\big)}{r^{(n-1)/2}\,(r+t)^{(n-1)/2}}\le C(n)\,,\qquad\forall r\in(-t,\infty)\,.
\end{equation}
Hence, by $p>1$, and provided $t$ is close enough to $-1$
\begin{eqnarray*}
\int_H(|x-t\,\ee|^{p'}-1)^{-n}\,dx&\le&C\,\int_2^\infty\frac{r^{(n-1)/2}\,(r+t)^{(n-1)/2}}{(r^{p'}-1)^n}\,dr+
C\,\int_{-t}^2\frac{r^{(n-1)/2}\,(r+t)^{(n-1)/2}}{(r^{p'}-1)^n}\,dr
\\
&\le&C\int_2^\infty\frac{r^{n-1}}{r^{np'}}\,dr+C\int_{-t}^2\frac{dr}{(r-1)^{n-(n-1)/2}}
\\
&\le &C\,\big(1+ |t+1|^{-(n-1)/2}\big)\le C\,|t+1|^{-(n-1)/2}\,.
\end{eqnarray*}
(We also notice that, by \eqref{lbbb}, one also has an analogous estimate from below, that is
\begin{equation}
  \label{vicky1}
  \int_H(|x-t\,\ee|^{p'}-1)^{-n}\,dx\ge c\,|t+1|^{-(n-1)/2}\qquad\mbox{for $|t+1|$ small enough}
\end{equation}
as well as
\begin{equation}
  \label{vicky2}
  \int_H(|x-t\,\ee|^{p'}-1)^{-(n-1)}\,dx\le C\,|t+1|^{-(n-3)/2}\qquad\mbox{for $|t+1|$ small enough}\,.
\end{equation}
Both estimates will be used in the last step of the proof of Theorem \ref{thm: properties of Phi}.) By combining this last estimate with \eqref{ts1} we find that
\begin{equation}
  \label{lim 1 1}
  T_{BE}(t)\ge c\,\Big(|t+1|^{-(n-1)/2}\Big)^{1/\psharp-1/\pstar}=c\,|t+1|^{-1/2\pstar}\,,
\end{equation}
for every $t$ close enough to $-1$, where $c=c(n,p)>0$. This proves that $T_{BE}(t)\to +\infty$ as $t\to-1$. Analogously,  again with $\e=|t+1|$,
\begin{eqnarray*}
\int_H|\nabla U_{BE}(x-t\,\ee)|^p&\ge &c\,\int_{H\cap B_{|t|+\e}(t\,\ee)}\,(|x-t\,\ee|^{p'}-1)^{-n}\,|x-t\,\ee|^{p'}\,dx
\\
&=&c\int_{-t}^{-t+\e}\frac{(r^2-t^2)^{(n-1)/2}r\,^{p'}}{(r^{p'}-1)^{n}}\,dr
\end{eqnarray*}
so that, setting $r=|t|+s\,|t+1|$, noticing that $r^2-t^2\ge c\,s\,|t+1|$ and $1\le r^{p'}\le 1+C\,|t+1|$, we get
 \begin{eqnarray*}
\int_H|\nabla U_{BE}(x-t\,\ee)|^p&\ge &|t+1|^{(n-1)/2}\int_{0}^{1}\frac{s^{(n-1)/2}|t+1|\,ds}{|t+1|^n}\ge \frac{c}{|t+1|^{(n-1)/2}}\,.
\end{eqnarray*}
Hence,
\begin{equation}
  \label{lim 1 2}
  G_{BE}(t)\ge c\,\Big(|t+1|^{-(n-1)/2}\Big)^{(1/p)-(1/\pstar)}=c\,|t+1|^{-(n-1)/2n}\,,
\end{equation}
and
\[
\lim_{p\to -1^+} G_{BE}(t) = \infty\,.
\]
(We also notice, again for future use in the proof of Theorem \ref{thm: properties of Phi}, that together with \eqref{lim 1 2} we also have
\begin{equation}
  \label{vicky4}
  G_{BE}(t)\le C(n)\,|t+1|^{-(n-1)/2n}\,,
\end{equation}
provided $t$ is close enough to $-1$.) Statement (iii) and (iv) follow immediately.
\end{proof}

\begin{proof}
  [Proof of Theorem \ref{thm: main}] Immediate from Theorem \ref{thm: equality cases} and Proposition \ref{prop: T and G}.
\end{proof}


We now turn to the quantitative study of $\Phi(T)$. Let us recall that, by a classical variational argument, if $u$ is a minimizer in $\Phi(T)$, then there exists constants $\lambda$ and $\s$ such that
 \begin{equation}\label{eqn: EL eqn}
 \begin{cases}
 -\Delta_p u = \lambda |u|^{\pstar-2}u & \text{ in } \Ha\\
 -|\n u |^{p-2} \partial_{x_1} u = \sigma |u|^{\psharp - 2}u & \text{ on } \pa \Ha \,.
 \end{cases}
 \end{equation}
Observe that the existence of constants $\lambda$ and $\s$ satisfying \eqref{eqn: EL eqn} follows by direct computation using our characterization of minimizers. Moreover, we know that non-negative minimizers are positive, so that there is no need for the absolute values in \eqref{eqn: EL eqn}.

\begin{lemma} \label{lem: derivative of Phi} If $n\geq 2,$ $1\leq p <n$, $T \in (0, \infty)$ and $\lambda$ and $\s$ are the Lagrange multipliers appearing in \eqref{eqn: EL eqn} and corresponding to a minimizer $u$ in the variational problem $\Phi(T)$, then the following identities hold:
\begin{equation}\label{eqn: derivative identity}
\Phi(T)^p=\lambda+\s\,T^\psharp\qquad
\Phi'(T) = \frac{\psharp T^{\psharp -1} }{\Phi(T)^{p-1}}\,\s\,.
\end{equation}
\end{lemma}

\begin{proof} The first identity follows from an integration by parts and \eqref{eqn: EL eqn}, so we focus on the second one. Since $T>0$ implies  $\int_{\pa H}u^\psharp>0$, there must be a function $\vphi \in C^{\infty}_c(\pa H)$ such that
	\begin{equation}
    \label{phi 1}
    \int_{\pa H} u^{\psharp -1} \vphi \, d \Hi^{n-1} = 1\,.
    \end{equation}
    Similarly, there exists $\xi\in C^\infty_c(H)$ such that
    \[
    \int_H\,u^{\pstar-1}\,\xi=1\,.
    \]
    Let $\psi$ be any function $\psi\in C^\infty_c(\ov{H})$ with $\psi=\vphi$ on $\pa H$, and extend $\vphi$ to $H$ by setting
    \[
    \vphi=\psi-\Big(\int_H\,u^{\pstar-1}\,\psi\Big)\,\xi\,.
    \]
    Then $\vphi \in C^\infty_c(\overline{H})$ and
    \begin{equation}
      \label{phi 2}
          \int_H u^{\pstar - 1} \vphi = 0\,.
    \end{equation}
    Now define a function $f:\R^2\to[0,\infty)$ by setting
    \[
    f(\e, \delta) =-1+ \int_H |u + \e \vphi + \delta \xi|^\pstar\qquad (\e,\de)\in\R^2\,.
    \]
    Since $u>0$ on $\ov{H}$, there exists a neighborhood $\U$ of $(\e,\de)=(0,0)$ such that $u + \e \vphi + \delta \xi>0$ on $\ov{H}$ for every $(\e,\de)\in\U$. Correspondingly, by \eqref{phi 2}
    \[
    f\in C^1(\U)\qquad f(0,0)\qquad \frac{\pa f}{\pa \de}(0,0)=\pstar\,\int_H\,u^{\pstar-1}\xi=1\,,
    \]
    and thus there exists $\e_0>0$ and $g:(-\e_0,\e_0)\to\R$ such that$(\e,g(\e))\in\U$ and $f(\e,g(\e))=0$ for every $|\e|<\e_0$. In particular,
     \[
    v_\e = u + \e \vphi + \g(\e)\,\xi\in C^\infty(\ov{H};(0,\infty))\qquad\int_H\,v_\e^\pstar=1\,,\qquad\forall |\e|<\e_0\,.
    \]
    By \eqref{phi 1}
    \begin{equation}\label{eqn: constraint}
	\frac{d}{d\e}\bigg|_{\e=0} \int_{\pa H } \frac{v_\e^\psharp}{\psharp} \, d\Hi^{n-1}=\int_{\pa H}u^{\psharp-1}\,\vphi=1\,,
    \end{equation}
    so that the function $\tau(\e)=\|v_\e\|_{L^\psharp(\pa H)}$ satisfies $\tau(0)=T$ and is strictly increasing on $(-\e_0,\e_0)$, up to possibly decreasing the value of $\e_0$. If we set $\G(\e)=\int_{H}|\nabla v_\e|^p$, then, by construction, $\Phi(\tau(\e))^p\le \Gamma(\e)$ for every $|\e|<\e_0$, with equality at $\e=0$, and thus
    \begin{equation}
      \label{s}
      	\frac{d}{d\e}\bigg|_{\e=0} \Phi(\tau(\e))^p=\frac{d}{d\e}\bigg|_{\e=0} \G(\e).
    \end{equation}
	We compute that
	\begin{equation}
      \label{s1}
  	 \frac{1}{p}\frac{d}{d\e}\bigg|_{\e=0} \G(\e) =
     \int_{H} |\n u|^{p-2} \n u \cdot \n \vphi \, dx = - \int_{H} \Delta_p u \vphi  - \int_{\pa H} |\n u |^{p-2} \pa_{x_1} u \vphi \, d\Hi^{n-1}.
    \end{equation}
	From \eqref{eqn: EL eqn}, $-\Delta_p u \vphi = \lambda u^{\pstar -1}$, and so the first term on the right-hand side of \eqref{s1} is equal to zero. Then, from \eqref{eqn: EL eqn} and \eqref{eqn: constraint}, the right-hand side of \eqref{s1} is equal to $\sigma$, and thus that of \eqref{s} to $p\,\s$. Since, again by \eqref{phi 1},
    \[
    \tau'(0)=\frac{T^{1-\psharp}}{\psharp}\,,
    \]
    we conclude from \eqref{eqn: constraint} that
    \[
    \Phi(T)^{p-1}\Phi'(T) \frac{T^{1-\psharp}}{\psharp}=\sigma\,,
    \]
    thus completing the proof of the lemma.
\end{proof}

We now prove Theorem \ref{thm: properties of Phi}, Corollary \ref{corollary bestbest} and Theorem \ref{thm: p=1}.

\begin{proof}[Proof of Theorem~\ref{thm: properties of Phi}]
{\it Step 1: Differentiability and monotonicity.} By Proposition \ref{prop: T and G} we know that $\Phi(T_{BE}(t))=G_{BE}(t)$ for every $t\in(-\infty,-1)$, where $T_{BE}$ is smooth and strictly increasing on $(-\infty,-1)$ with range $(T_E, \infty)$ and $G_{BE}(t)$ is smooth and strictly increasing on $(-\infty,-1)$ with range $(G_E, \infty)$. Thus, $\Phi$ is smooth on $(T_E,\infty)$ with $\Phi'(T)=G_{BE}'(t)/T_{BE}'(t)>0$ for $T=T_{BE}(t)$. This shows that $\Phi$ is smooth and strictly increasing on $(T_E,\infty)$. One can compute that
\[
\lim_{T\to T_E^+}\Phi'(T)=\lim_{t\to -\infty}\,\frac{G_{BE}'(t)}{T_{BE}'(t)}=\frac{G_E}{T_E}=
E\,.
\]
Similarly, $\Phi(T_S(t))=G_S(t)$ for every $t\in\R$ where $T_S$ is strictly decreasing on $\R$ with range $(0, T_E)$, and $T_S(0)=T_0<T_E$, and where $G_S$ is strictly increasing on $[0,\infty)$ with range $[2^{-1/n}S, S)$, and is strictly decreasing on $(-\infty,0)$ with range $(2^{-1/n}S,E)$. Hence $\Phi$ is smooth on $(0,T_E)$, and $\Phi'(T)=G_S'(t)/T_S'(t)>0$, $T=T_S(t)$, so $\Phi$ is strictly decreasing on $(0,T_0)$ and strictly increasing on $(T_0,T_E)$, and one computes
\[
\lim_{T\to T_E^-}\Phi'(T)=\lim_{t\to -\infty}\,\frac{G_S'(t)}{T_S'(t)}=\frac{G_E}{T_E}=E\,.
\]
Hence $\Phi$ is differentiable also at $T=T_E$, and thus on $(0,\infty)$.

\medskip

\noindent{\it Step 2: Concavity of $\Phi.$}
Next, we use Lemma~\ref{lem: derivative of Phi} to show that  $\Phi$ is concave for $T$ sufficiently small. By \eqref{eqn: derivative identity} we find that for every $t\in\R$,
\begin{equation}
  \label{hhh}
  \frac{d}{dT}\Phi(T_S(t))=\psharp\frac{T_S(t)^{\psharp-1}}{G_S(t)^{p-1}}\,\s(U_S(t))\,,
\end{equation}
where $\s(U_S(t))$ denotes the boundary Lagrange multiplier of $U_S(t)$. By combining \eqref{def of USt} and \eqref{eqn: EL eqn}, we see that
\[
\sigma(U_S(t))=-c(n,p)\,t\,\|U_S\|_{L^\pstar(\{x_1>t\})}^{p(p-1)/(n-p)}\,,
\]
for a positive constant $c(n,p)$. Thus,
\[
\frac{d}{dT}\Phi(T_S(t))=-c(n,p)\,t\,\frac{\|U_S\|_{L^{\psharp}(\{x_1=t\})}^{\psharp-1}}{\|\nabla U_S\|_{L^p(\{x_1>t\})}^{p-1}}\,,
\]
so, differentiating in $t$ (recall that $\Phi$ is smooth $(0,T_E)$), we find that
\[
\frac{d^2}{dT^2}\Phi(T_S(t))\,T_S'(t)=-c(n,p)\frac{d}{dt}\bigg(t\,\frac{\|U_S\|_{L^{\psharp}(\{x_1=t\})}^{\psharp-1}}{\|\nabla U_S\|_{L^p(\{x_1>t\})}^{p-1}}\bigg)\,.
\]
Since $T_S'(t)<0$ for every $t\in\R$, we conclude that $\Phi(T)$ is going to be concave on any interval $J=\{T_S(t):t\in J'\}$ corresponding to an interval $J'\subset\R$ such that
\begin{equation}
  \label{holds}
  \frac{d}{dt}\log\bigg(t\,\frac{\|U_S\|_{L^{\psharp}(\{x_1=t\})}^{\psharp-1}}{\|\nabla U_S\|_{L^p(\{x_1>t\})}^{p-1}}\bigg)<0\,,\qquad\forall t\in J'\,.
\end{equation}
For the sake of brevity, set
\begin{eqnarray*}
h(t)=\int_{\{x_1=t\}}U_S^\psharp =\int_{\pa H} (1+|x- t\,\ee|^{p'})^{-(n-1)}\, d \Hi^{n-1}\,.
\end{eqnarray*}
We are thus looking for an interval $J'$ such that
\[
\frac1t+\frac{\psharp-1}{\psharp}\,\frac{h'(t)}{h(t)}-\frac{p-1}p\,\frac{\frac{d}{dt}\int_{\{x_1>t\}}|\nabla U_S|^p}{\int_{\{x_1>t\}}|\nabla U_S|^p}<0\qquad\forall t\in J'\,.
\]
Since $\int_{\{x_1>t\}}|\nabla U_S|^p$ is trivially increasing in $t$, it suffices to find an interval $J'$ such that
\[
\frac1t+\frac{n(p-1)}{p(n-1)}\,\frac{h'(t)}{h(t)}<0\qquad\forall t\in J'\,.
\]
If $t>0$, then factoring and changing variables, we find that
\[
h(t)= t^{-(n-1)/(p-1)} \int_{\pa H} (t^{-p'}+ |x-\ee|^{p'})^{-(n-1)} \, d\Hi^{n-1}\,.
\]
Therefore, we compute
\[
\frac{h'(t)}{h(t)} = -\frac{n-1}{(p-1)t}  +  \frac{p'(n-1)\int_{\pa H} (t^{-p'} + |x+\ee|^{p'})^{-n} \, d\Hi^{n-1} }{t^{p' +1}\int_{\pa H} (t^{-p'} + |x+\ee|^{p'})^{-(n-1)} \, d\Hi^{n-1} }\,,
\]
where trivially $t^{-p'} + |x+\ee|^{p'}>1$ for $x\in\pa H$, and thus
\[
\int_{\pa H} (t^{-p'} + |x+\ee|^{p'})^{-n} \, d\Hi^{n-1}<\int_{\pa H} (t^{-p'} + |x+\ee|^{p'})^{-(n-1)} \, d\Hi^{n-1}\,.
\]
We have thus proved that for every $t>0$,
\[
\frac{h'(t)}{h(t)}
 \leq\frac1{t}\, \frac{n-1}{p-1} \Big(-1+\frac{p}{t^{p'}}\Big)\,,
\]
so that
\[
\frac1t+\frac{n(p-1)}{p(n-1)}\,\frac{h'(t)}{h(t)}\le \frac1t+\frac1{t}\, \frac{n}{p} \Big(-1+\frac{p}{t^{p'}}\Big)\,.
\]
This last quantity is negative for $t>(\pstar)^{1/p'}$. Thus, \eqref{holds} holds with the choice
\[
J'=(-\infty,(\pstar)^{1/p'})
\]
and correspondingly $\Phi(T)$ is strictly concave on $(0,T_*)$ provided we set
\[
T_*=T_S((\pstar)^{1/p'})\,.
\]

\medskip

\noindent{\it Step 3: Convexity of $\Phi$.} By \eqref{eqn: consequence of mother BE} we have that for every $t<-1$
\begin{equation}
  \label{eqn: used for convexity}
  \Phi(T)\ge G_{BE}(t)+t\,\frac{T_{BE}(t)^\psharp-T^\psharp}{\psharp\,Y_{BE}(t)}\qquad\forall T>0\,,
\end{equation}
with equality if and only if $T=T_{BE}(t)$. If we denote by $\Psi_t(T)$ the right-hand side of \eqref{eqn: used for convexity}, this shows that
\[
\Phi(T)=\sup_{t<-1}\Psi_t(T)\qquad \forall T\in\big\{T_{BE}(t):t<-1\}=(T_E,\infty)\,.
\]
Since each $\Psi_t(T)$ is convex as a function of $T$ (recall that $t$ is negative), this proves that $\Phi(T)$ is convex on $(T_E,\infty)$.
We can perform the same argument based on \eqref{eqn: consequence of mother}, as soon as the parameter $t\in\R$ describing the Sobolev family is negative. This proves the convexity of $\Phi(T)$ over the interval
\[
\big\{T_S(t):t<0\big\}=(T_0,T_E)\,.
\]
Since $\Phi(T)$ is convex on $(T_0,T_E)$ and on $(T_E,\infty)$, with $\Phi(T)\ge E\,T$ for every $T\ge 0$ and $\Phi(T_E)=E\,T_E$, we conclude that $\Phi(T)$ is convex on $(T_0,\infty)$.

\medskip

\noindent{\it Step 4: Asymptotic growth of $\Phi$.} First, we claim that
\[
\lim_{T\to\infty}\frac{\psharp\,\Phi(T)}{T^\psharp}=1\,.
\]
Having in mind \eqref{global lower bound}, and taking into account that $T_{BE}(t)\to+\infty$ as $t\to -1$, it suffices to show that
\[
\lim_{t\to -1}\frac{\psharp G_{BE}(t)}{T_{BE}(t)^\psharp}=1\,.
\]
To prove this, we notice that the identity
\begin{equation}
  \label{idv}
  n\int_{H}U_{BE,t}^\psharp=\psharp\,G_{BE}(t)\,Y_{BE}(t)+t\,T_{BE}(t)^\psharp\qquad\forall t<-1\,,
\end{equation}
allows us to write
\[
\psharp\,\frac{G_{BE}(t)}{T_{BE}(t)^\psharp}=-\frac{t}{Y_{BE}(t)}+\frac{n\int_{H}U_{BE,t}^\psharp}{Y_{BE}(t)\,T_{BE}(t)^\psharp}\,.
\]
It will thus be enough to prove
\begin{equation}
  \label{vicky3}
  \lim_{t\to -1}Y_{BE}(t)=1\qquad \lim_{t\to -1}\int_{H}U_{BE,t}^\psharp=0\,.
\end{equation}
To this end, we first notice that by \eqref{vicky1} and \eqref{vicky2}
\begin{eqnarray*}
  Y_{BE}(t)^{p'}-1=\frac{\int_H(|x-t\,\ee|^{p'}-1)^{-(n-1)}}{\int_H(|x-t\,\ee|^{p'}-1)^{-n}}\le C(n)\,\frac{|t+1|^{-(n-3)/2}}{|t+1|^{-(n-1)/2}}=C(n)\,|t+1|\,,
\end{eqnarray*}
while
\begin{eqnarray*}
  \int_{H}U_{BE,t}^\psharp=
  \frac{\int_H(|x-t\,\ee|^{p'}-1)^{-(n-1)}}{\big(\int_H(|x-t\,\ee|^{p'}-1)^{-n}\big)^{(n-1)/n}}
  \le C(n)\,\frac{|t+1|^{-(n-3)/2}}{|t+1|^{-(n-1)^2/2n}}=C(n)\,|t+1|^{(n+1)/2n}
\end{eqnarray*}
so that \eqref{vicky3} is proven. Now, to prove that
\[
\lim_{T\to\infty}\Phi(T)-\frac{T^\psharp}{\psharp}=0
\]
we simply notice that, again by \eqref{idv},
\[
\psharp G_{BE}(t)-T_{BE}(t)^\psharp=\psharp\,G_{BE}(t)\Big(1+\frac{Y_{BE}(t)}t\Big)-\frac{n}t\,\int_H U_{BE,t}^\psharp\,.
\]
Since $|t+Y_{BE}(t)|\le|t+1|+|1-Y_{BE}(t)|\le C(n)\,|t+1|$, thanks to \eqref{vicky4} we have
\[
G_{BE}(t)\Big|1+\frac{Y_{BE}(t)}t\Big|\le C(n)\,|t+1|^{1-(n-1)/2n}=C(n)\,|t+1|^{(n+1)/2n}\to 0\,.
\]
This completes the proof of Theorem \ref{thm: properties of Phi}.
\end{proof}

\begin{proof}[Proof of Corollary \ref{corollary bestbest}] Since $\Omega$ is a set of locally finite perimeter in $\R^n$ \cite[Example 12.6]{maggiBOOK}, there exists $x_0 \in \pa \Om $ such that, up to a rotation,
\begin{equation}\label{eqn: blowup conv}
\begin{split}
\Om_r\to H&\qquad\mbox{in $L^1_{{\rm loc}}(\R^n)$},
\\
\Hi^{n-1} \llcorner \pa \Om_r \weak \Hi^{n-1} \llcorner \pa H&\qquad\mbox{as Radon measures on $\R^n$}.
\end{split}
\end{equation}
where we have set $\Om_r=(\Om-x_0)/r$, $r>0$. Precisely, every $x_0$ in the reduced boundary of $\Om$ satisfies \eqref{eqn: blowup conv} up to a rotation, see e.g. \cite[Theorem 15.5]{maggiBOOK}.

We now define a function $w_T$ depending on $T$ as follows. If $T\in(0,T_E)$, setting $t=T_S^{-1}(T)$, we let
\[
w_T(x)=U_{S,t}(x)\qquad \forall x\in\R^n\,;
\]
if $T=T_E$, we set $t=-1$ and let
\[
w_T(x)=U_{E,t}(x)\qquad\forall x\in\R^n\setminus\{\ee\}\,;
\]
finally, if $T>T_E$, then, setting $t= T_{BE}^{-1}(T)<-1$, we let
\[
w_T(x)=U_{BE, t}(x)\qquad\forall x\in\R^n\setminus \ov{B_1(t\,\ee)}\,.
\]
Notice that in each case, there exists a compact set $K_T$ with $K_T\cap H=\emptyset$ such that $w_T\in L^\pstar(\R^n\setminus U)$ and $\nabla w_T\in L^p(\R^n\setminus U)$ for every open neighborhood $U$ of $K_T$. In particular, for $\e>0$ small enough depending on $T$, we have $\{x_1>-\e\}\cap K_T=\emptyset$. We pick $\zeta\in C^\infty(\R^n)$ such that $\zeta=1$ on $\{x_1>-\e\}$ and $\zeta=0$ on $K_T$, and define $v_T=\zeta\,w_T$ on the whole $\R^n$. Then
\begin{equation}
  \label{varie}
  v_T\in L^\pstar(\R^n)\qquad \nabla v_T\in L^p(\R^n)\qquad\mbox{$v_T=w_T$ on $H$}\,.
\end{equation}

Next, we fix $R>0$ and consider $\psi_R\in C^\infty_c(B_{2R};[0,1])$ with $\psi_R=1$ on $B_R$. Finally, for each $r>0$, we define
\[
u_r(x)= r^{1-n/p}\,(\psi_R\,v_T)\Big(\frac{x-x_0}{r} \Big)\qquad x\in\Omega\,.
\]
By \eqref{eqn: blowup conv}, \eqref{varie} and $\psi_R\in C^\infty_c(B_{2R};[0,1])$ we can exploit dominated convergence to find that
\begin{eqnarray*}
  &&\int_{\Om}u_r^\pstar=\int_{\R^n}1_{\Om_r}(\psi_R v_T)^\pstar\to\int_H(\psi_R v_T)^\pstar=\int_H(\psi_R w_T)^\pstar
  \\
  &&\int_\Om|\nabla u_r|^p\to\int_H|\nabla(\psi_R\,w_T)|^p\,,
\end{eqnarray*}
as $r\to 0^+$. Similarly, since $(\psi_R\,v_T)^\psharp\in C^0_c(\R^n)$, by \eqref{eqn: blowup conv} we have
\[
\int_{\pa\Om}u_r^\psharp \, d\Hi^{n-1}=\int_{\R^n}(\psi_R\,v_T)^\psharp\,d(\H^{n-1}\llcorner\pa\Om_r)
\to\int_{\pa H}(\psi_R\,v_T)^\psharp\,d\H^{n-1}=\int_{\pa H}(\psi_R\,w_T)^\psharp\,d\H^{n-1}
\]
as $r\to 0^+$. Since
\[
\int_H w_T^\pstar=1\,,\qquad \int_{\pa H}w_T^\psharp=T^\psharp\,,\qquad\int_H|\nabla w_T|^p=\Phi(T)^p\,,
\]
for every $\de>0$ there exists $r$ small enough and $R$ large enough such that
\[
\Big|\int_{\Om}u_r^\pstar-1\Big|+\Big|\int_{\Om}|\nabla u_r|^p-\Phi(T)^p\Big|+\Big|\int_{\pa\Om}u_r^\psharp-T^\psharp\Big|<\de\,.
\]
In particular, we can find $\{\xi_r\}_{r>0} \subset C^{\infty}_c(\Om)$ such that
\[
\frac{\| u_r+ \xi_r\|_{L^\psharp(\pa \Om)}}{\| u_r + \xi_r\|_{L^\pstar(\Om)} } = \frac{\| u_r\|_{L^\psharp(\pa \Om)}}{\| u_r + \xi_r\|_{L^\pstar(\Om)} }= T\qquad\forall r>0\,,
\]
and $\| \xi_r\|_{L^\pstar(\Om)} \to 0$ and $\| \n \xi_r\|_{L^p(\Om)} \to 0$ as $r\to 0^+$. Then, for $r$ sufficiently small,
\[
\Phi_\Om(T)
\le \frac{\|\nabla u_r + \n \xi_r \|_{L^p(\Om)}}{\|u_r + \xi_r \|_{L^\pstar(\Om)}}\le (1+C\,\de)\,\Phi(T)
\]
for a constant $C=C(n,p)$.
\end{proof}

\begin{proof}[Proof of Theorem~\ref{thm: p=1}]
With the same reasoning as given in Corollary~\ref{cor: cor of mother}, we find that
\[
Y_S(t) G_S(t) + tT_S(T) \leq Y_S(t) |Dh|(H) + t \| h\|_{L^1(\pa H)}
\]
 for any $t \in (-1,1)$ and any non-negative $h$, vanishing at infinity, with $|Dh|(H)<\infty$, and with equality if and only if $h$ is a dilation translation image of $U_{S,t}$ orthogonal to $\ee$. In particular, if additionally $\|h \|_{L^1(\pa H)} = T_S(t),$ then
 \[
 G_S(t) \leq |Dh|(H).
 \]
 From this, we deduce that
 \[
 \Phi(T_S(t)) = G_S(t)
 \]
 for $t \in (-1,1).$ The same arguments given in the proof of Proposition ~\ref{prop: T and G} imply that $T_S(t)$ is a strictly decreasing function with range $[0,\infty)	$, and that $G_S(t)$ is strictly increasing for $t >0$ with range $( 2^{-1/n}S,S) $ and is strictly decreasing for $t<0$ with range $(2^{-1/n}S, \infty)$. Finally, the same proof as that of Theorem~\ref{thm: properties of Phi} shows that $\Phi(T)$ is a smooth function of $T$ that is decreasing for $T \in (0, T_0)$
  and concave for $T\in (0, T_*)$ for some $0<T_*<T_0$
   and increasing and convex for $T \in (T_0, \infty)$.
 Finally, to show that  $\Phi(T) = T + o(1)$ as $T\to \infty$, we will equivalently show that $G_S(t) =T_S(t) + o(1)$ as $t \to -1$. Indeed, since $Y(t,U_{S,t}) =1$ for all $-1<t<1$ when $p=1$, \eqref{eqn: eq other} implies that
 \[
 G_{S}(t) = n\int_H U_{S,t} -t\, T_{S}(t)=  T_{S}(t) +n\int_H U_{S,t} -(t+1)\, T_{S}(t) .
 \]
 Note that
 \[
 \int_H U_{S,t} = \frac{|B_1(t\,\ee) \cap H|}{|B_1(t\,\ee) \cap H|^{(n-1)/n} } = |B_1(t\,\ee) \cap H|^{1/n} = o(1)
 \]
 as $t \to -1$.
 Furthermore, since
 \[
 T(t)  =\frac{ \om_{n-1} (1-t^2)^{(n-1)/2}}{|B_1(t\,\ee) \cap H|^{(n-1)/n}},
 \]
 and we easily estimate that
 \[
 |B_1(t\,\ee) \cap H| = \om_{n-1} \int_{-t}^1 (1-s^2)^{(n-1)/2}\,ds \geq c \int_{-t}^1 (1-s)^{(n-1)/2} \geq C |1+t|^{(n+1)/2}
 \]
 for $t<0$, we see that
 \[
|t+1| T(t) \leq |t+1|^{1-(n+1)/2n} = o(1).
 \]
Hence, $G_S(t) =T_S(t) + o(1)$ as $t \to -1$ and the proof is complete.
\end{proof}

We conclude with the following proposition, which was mentioned in the comments after the statement of Theorem \ref{thm: p=1}.

\begin{proposition}\label{prop TEnp}
 For every $n\ge 2$, one has $T_E(n,p)\to+\infty$ as $p\to 1^+$.
\end{proposition}

\begin{proof}
As a first step, we explicitly compute
\begin{equation}\label{eqn: TE explicit}
T_E^\psharp
=C\, \Bigg(
\frac{\Gamma\big(\frac{n-1}{2(p-1)}\big)}{\Gamma\big(\frac{(n-1)p}{2(p-1)}\big)}\Bigg)\Bigg/ \Bigg((p-1) \frac{ \Gamma\big(\frac{n+p-1}{2(p-1)}\big)}{\Gamma\big(\frac{np}{2(p-1)} \big)}\Bigg)^{(n-1)/n},
\end{equation}
where, here and throughout the proof, $C$ denotes a constant depending only on $n$, whose value may change at each instance.
Indeed,
\begin{align*}
\int_{\pa H}|x+\ee |^{-(n-1)p'} \, d\Hi^{n-1} &= \int_{\R^{n-1}} (|z|^2 +1)^{-(n-1)p'/2} \, dz = C\,\int_0^\infty (r^2 +1 )^{-(n-1)p'/2} r^{n-2} \, dr.
\end{align*}
Making the change of variables $s = 1/(r^2+1)$, the right-hand side becomes
\begin{align*}
C \int_0^1 s^{[(n-1)p'/2]-2} \left(1/s-1\right)^{(n-3)/2} \, ds
&=C \int_0^1 s^{[(n-1)/2(p-1)]-1} (1-s)^{(n-3)/2} \, ds\\
 = C\,\B\Big( \frac{n-1}{2(p-1)} , \frac{n-1}{2}\Big)&= C \,\Gamma\Big(\frac{n-1}{2(p-1)}\Big)\Big/\Gamma \Big(\frac{(n-1)p}{2(p-1)}\Big).
\end{align*}
To express the term in the denominator of $T_E$, the coarea formula implies that
\[
\int_{H}|x+\ee|^{-np'} \, dx  = \int_1^{\infty} r^{-np' + (n-1)} \int_{1/r}^1 (1-s^2)^{(n-3)/2} \,ds \,dr.
\]
By Fubini's Theorem, the right-hand side is equal to
\[
\int_0^1 (1-s^2)^{(n-3)/2} \int_{1/s}^\infty r^{-[n/(p-1)] -1} \, dr \, ds = \frac{p-1}{n} \int_0^1 (1-s^2)^{(n-3)/2}s^{n/(p-1)}\,ds\,.
\]
With the change of variables $\rho = s^2,$ this is equal to
\begin{align*}
\frac{p-1}{2n} \int_0^1 (1-\rho)^{(n-3)/2} \rho^{[n/2(p-1)] -1/2}\,d\rho
&=  \frac{p-1}{2n}\, \B\Big( \frac{n-1}{2}, \frac{n}{2(p-1) }+\frac{1}{2}\Big) \\
& = C (p-1)\, \Gamma \Big(\frac{n+p-1}{2(p-1)}\Big)\Big/\Gamma\Big(\frac{np}{2(p-1)}\Big).
\end{align*}
This proves \eqref{eqn: TE explicit}. By taking the logarithm of $T_E^\psharp/C$, we find that
\begin{equation}\label{eqn: log term}
\begin{split}	
\log (T_E^\psharp/C)
& =  \log \Gamma\Big(\frac{n-1}{2(p-1)}\Big)
	 -\log \Gamma\Big(\frac{p(n-1)}{2(p-1)}\Big) \\
& -\frac{n-1}{n}\left[ \log(p-1) + \log \Gamma\Big(\frac{n+p-1}{2(p-1)} \Big) - \log \Gamma\Big(\frac{np}{2(p-1)} \Big)\right].
\end{split}
\end{equation}
By Stirling's approximation, $\log \Gamma(z)$ asymptotically behaves like $ z \log(z)$ as $z \to \infty$. Hence, in the limit $p\to 1^-$ the first two terms on the right-hand side of \eqref{eqn: log term}, behave like
\begin{align*}
\frac{n-1}{2(p-1)} &\log\Big(\frac{n-1}{2(p-1)}\Big) - \frac{p(n-1)}{2(p-1)}\log\Big(\frac{p(n-1)}{2(p-1)}\Big)\\
&= \frac{(n-1)}{2(p-1)}\left[\log \Big(\frac{n-1}{2(p-1)}\Big) - \log\Big( \frac{(n-1)p}{2(p-1)}\Big)\right]
-\frac{(n-1)}{2}\log\Big( \frac{(n-1)p}{2(p-1)}\Big)
 \\
&= -\frac{(n-1)}{2(p-1)}\log(p)
-\frac{(n-1)}{2}\log\Big( \frac{(n-1)p}{2(p-1)}\Big) \\
&= -\frac{p(n-1)}{2(p-1)}\log(p)
+\frac{(n-1)}{2}\log\left(p-1\right)+C.
\end{align*}
On the other hand, the term in brackets on the right-hand side of \eqref{eqn: log term} behaves like
\begin{align*}
& \log(p-1) + \frac{n+p-1}{2(p-1)}\log\Big(\frac{n+p-1}{2(p-1)}\Big) - \frac{np}{2(p-1)}\log \Big(\frac{np}{2(p-1)}\Big)\\
&= \log(p-1)\Big( 1-\frac{n+p-1}{2(p-1)}+ \frac{np}{2(p-1)}\Big) + \frac{n+p-1}{2(p-1)}\log\Big(\frac{n+p-1}{2}\Big) - \frac{np}{2(p-1)}\log \Big(\frac{np}{2}\Big)\\
&= \log(p-1)\Big( \frac{n+1}{2}\Big) + \frac{n+p-1}{2(p-1)}\log\Big(\frac{n+p-1}{2}\Big) - \frac{np}{2(p-1)}\log \Big(\frac{n}{2}\Big) - \frac{np}{2(p-1)}\log(p)\\
&= \log(p-1)\Big( \frac{n+1}{2}\Big) + \frac{n+p-1}{2(p-1)}\log\Big(n+p-1\Big) - \frac{np}{2(p-1)}\log \left(n\right) - \frac{np}{2(p-1)}\log(p)+C.
\end{align*}
So, the full right-hand side of \eqref{eqn: log term} asymptotically behaves like
\[
-\frac{n-1}{2n}\log\left(p-1\right) +\frac{n-1}{2n(p-1)}\left[-(n+p-1)\log\left(n+p-1\right) + np\log \left(n\right)\right]
+C.
\]
Since $\log(n + p-1) = \log(n) + (p-1)/n  + o(p-1),$ this quantity is bounded above and below (with appropriate choices of $C$) by
\begin{align*}
&
-\frac{n-1}{2n}\log\left(p-1\right) +\frac{n-1}{2n(p-1)} \left[-(n+p-1)\Big(\log(n) +\frac{p-1}{n}\Big) + np\log \left(n\right)\right]
+C\\
&=-\frac{n-1}{2n}\log\left(p-1\right) +\frac{n-1}{2n(p-1)}\left[(n-1)(p-1)\log(n) -(n+p-1)\frac{p-1}{n} \right]+ C\\
&=-\frac{n-1}{2n}\log\left(p-1\right) + \frac{(n-1)}{2n}\left[(n-1)\log(n) -\frac{n+p-1}{n} \right]
+C
\end{align*}
The second term is bounded above and below by dimensional constants, while the first term goes to $+\infty$ as $p \to 1^+$.
\end{proof}

\appendix

\section{Proof of Theorem \ref{thm: equality cases}}\label{sec: appendix} In this appendix, we prove Theorem \ref{thm: equality cases}, which aims to characterize the equality cases in Theorem \ref{thm: mother}. The main step is to prove the validity of \eqref{eqn: formal IBP new } (see the proof of Theorem \ref{thm: mother}) without the assumption that $f\in C^1_c(\ov{H})$. This is the content of the following lemma, whose proof resembles \cite[Theorem 7]{cenv}.

\begin{lemma}\label{lem: IBP}
 If $n \geq 2$, $p\in [1,n)$, and $f$ and $g$ are non-negative functions in $L^1_{{\rm loc}}(H)$, vanishing at infinity, with
  \begin{equation}
    \label{hp on f and g APP}
    \left\{\begin{split}
      &\mbox{$\int_H|\nabla f|^p<\infty$ and $\int_{H}|x|^{p'}g^\pstar<\infty$ if $p>1$}
      \\
      &\mbox{$|Df|(H)<\infty$ and $\spt\,g\cc\ov{H}$ if $p=1$}
      \\
      &\| f\|_\norm = \|g\|_\norm=1
    \end{split}
    \right .
  \end{equation}
then \eqref{eqn: formal IBP new } holds for every $t\in\R$, that is
\begin{equation}
  \label{formal}
  n\int_{H}g^\psharp\le -\psharp\int_H\,f^{\psharp-1}\,\nabla f\cdot (T-t\,\ee)+t\,\int_{\pa H}f^\psharp\,,\qquad\forall t\in\R\,.
\end{equation}
Here $T=\nabla\vphi$ is the Brenier map from $f^\pstar\,dx$ and $g^\pstar dx$.
\end{lemma}

\begin{proof} We let $\Om$ be the interior of $\{\vphi<\infty\}$, and recall that $T\in (BV\cap L^\infty)_{{\rm loc}}(\Om;\R^n)$ with $F\,dx$ concentrated on $H\cap \Om$. We notice that in the proof of Theorem \ref{thm: mother}, see \eqref{eqn: AMGM inequality}, the identity
\begin{equation}
  \label{identity}
  \int_H\,g^\psharp=\int_H\,(\dett \nabla^2 \vphi)^{1/n} f^\psharp\,,
\end{equation}
was established without exploiting the additional assumption $f\in C^1_c(\ov{H})$. Thus \eqref{identity} also holds in the present setting.

We first let $p\in(1,n)$. By a translation orthogonal to $\ee$, we may assume that $0 \in \Om$. For $\e>0$ let $\eta_\e\in C^\infty_c(B_{2/\e};[0,1])$ with $\eta_\e=1$ on $B_{1/\e}$ and $\eta_\e\uparrow1$ pointwise on $\R^n$ as $\e\to 0^+$, and set
\[
f_\e (x) = \min \Big\{ f\Big( \frac{x}{1-\e}\Big), f(x) \eta_\e ( x) \Big\}\ind_{H_\e}(x),\qquad x\in H\,,
\]
where $H_\e=\{x_1>\e\}$. By density of $C^0_c(H)$ into $L^\pstar(H)$ we see that $f\circ((1-\e)^{-1}\Id)\to f$ in $L^\pstar(H)$ as $\e\to 0^+$, so that $f_\e\to f$ in $L^\pstar(H)$. Analogously, $\nabla[f\circ((1-\e)^{-1}\Id)]\to\nabla f$ in $L^p(H)$ as $\e\to 0^+$. If we choose $\eta_\e(x)=\eta(\e\,x)$ for some fixed $\eta\in C^1_c(B_2;[0,1])$ with $\eta=1$ on $B_1$, then we find
\[
\int_H|f\nabla\eta_\e|^p\le\Big(\int_{\R^n\setminus B_{1/\e}}f^\pstar\Big)^{p/\pstar}\,\Big(\int_{B_2}|\nabla\eta|^n\Big)^{p/n}
\to 0\qquad\mbox{as $\e\to 0^+$}\,,
\]
and thus $\nabla(f\,\eta_\e)\to\nabla f$ in $L^p(H)$. Finally, $\int_{H\setminus H_\e}|\nabla f|^p\to 0$ as $\e\to 0^+$, so that
\begin{equation}
  \label{convergence eps}
  \left\{\begin{split}
    &\mbox{$f_\e\to f$ in $L^\pstar(H)$ and a.e. on $H$}
    \\
    &\mbox{$1_{H_\e}\,\nabla f_\e\to\nabla f$ in $L^p(H)$}
  \end{split}\right .\qquad\mbox{as $\e\to 0^+$}\,.
\end{equation}
Moreover, as  $0\in\Om$ and $f=0$ a.e. on $\Om^c$, there exists an open set $\Om_\e\cc\Om$ such that $\supp (f_\e)\cc\Om_\e$. We can thus find $\{f_{\e,k}\}_{k\in\N}\subset C^1_c(\Om_\e\cap \ov{H_\e})$ such that
\begin{equation}
  \label{convergence}
  \left\{
  \begin{split}
    &\mbox{$f_{\e,k}\to f_\e$ in $L^\pstar(H_\e)$ and a.e. on $H_\e$}
    \\
    &\mbox{$\nabla f_{\e,k}\to\nabla f_\e$ in $L^p(H_\e)$}
  \end{split}
   \right .\qquad\mbox{as $k\to\infty$}\,.
\end{equation}
Since $f_{\e,k}\in C^1_c(\ov{\He})$, arguing as in Theorem~\ref{thm: mother} we find that
\begin{equation}\label{eqn: formal IBP epsilon}
 n\int_{H_\e}\,(\dett \nabla^2 \vphi)^{1/n} f_{\e,k}^\psharp\leq - \psharp \int_{\He} f_{\e,k}^{\psharp -1} \n f_{\e,k} \cdot S dx  + t
 \int_{\pa {\He}} f_{\e, k}^{\psharp} \, d\Hi^{n-1}\,
 \end{equation}
where $S=T-t\,\ee\in L^\infty_{{\rm loc}}(\Om;\R^n)$. Since $S$ is bounded on $\Om_\e$, where the $f_{\e,k}$ are uniformly supported in, and since $\psharp-1=\pstar/p'$, by \eqref{convergence} we find
\[
\lim_{k\to\infty}\int_{\He} f_{\e,k}^{\psharp -1} \n f_{\e,k} \cdot S dx=\int_{\He} f_\e^{\psharp -1} \n f_\e \cdot S dx\,.
\]
Moreover, by the trace inequality
\[
\|u\|_{L^\psharp(\pa A)}\le C(A)\,\Big(\|\nabla u\|_{L^p(A)}+\|u\|_{L^1(A)}\Big)\, ,
\]
which is valid whenever $A$ is an open bounded Lipschitz set (see, for example, \cite{MV2005}), and again by the uniform support property, \eqref{convergence} implies
\[
\lim_{k\to\infty}\int_{\pa {\He}} f_{\e, k}^{\psharp} \, d\Hi^{n-1}=\int_{\pa {\He}} f_\e^{\psharp} \, d\Hi^{n-1}\,.
\]
Hence, by pointwise convergence and Fatou's lemma, \eqref{eqn: formal IBP epsilon} implies
\begin{equation}\label{eqn: formal IBP epsilon xx}
 n\int_{H_\e}\,(\dett \nabla^2 \vphi)^{1/n} f_\e^\psharp\leq - \psharp \int_{\He} f_\e^{\psharp -1} \n f_\e \cdot S + t
 \int_{\pa {\He}} f_\e^{\psharp} \, d\Hi^{n-1}\,.
 \end{equation}
In order to take the limit $\e\to 0^+$ in \eqref{eqn: formal IBP epsilon xx}, we first notice that $f_\e\le f$ everywhere on $H$. Hence, by \eqref{eqn: integral push forward} and \eqref{hp on f and g APP}, we find
\[
\int_{H}|f_\e^{\psharp-1}\,S|^{p'}\le \int_{H}f^{\pstar}\,|S|^{p'}=
\int_{H}g^\pstar\,|x-t\,e_1|^{p'}<\infty\,.
\]
Since $f_\e\to f$ a.e. on $H$, it must be $f_\e^{\psharp -1}S\rightharpoonup f^{\psharp-1}S$ in $L^{p'}(H)$ as $\e\to0^+$. By combining this last fact with the strong convergence $1_H\,\nabla f_\e\to \nabla f$ in $L^p(H)$, we conclude that
\begin{equation}
  \label{second xx}
  \int_{\He} f_\e^{\psharp -1} \n f_\e \cdot S dx=\int_H f_\e^{\psharp -1} \n f_\e \cdot S \to \int_H f^{\psharp -1} \n f \cdot S
\end{equation}
as $\e\to 0^+$. Next, let us set $h_\e(x)=f_\e(x+\e\,\ee)$ for $x\in H$, so that $1_{H_\e}\nabla f_\e\to\nabla f$ in $L^p(H)$ and the density of $C^0_c(H)$ in $L^p(H)$ gives us $\nabla h_\e\to\nabla f$ in $L^p(H)$. By applying \eqref{escobar inequality lp} to $h_\e-f$ we find that $h_\e\to f$ in $L^\psharp(H)$, which clearly implies
\[
\lim_{\e\to 0^+}\int_{\pa H_\e}f_\e^\psharp\,d\H^{n-1}=\int_{\pa H}f^\psharp\,d\H^{n-1}\,.
\]
By combining this last fact with \eqref{second xx} with the fact that $1_{H_\e}f_\e^{\psharp}\to 1_H\,f^\psharp$ a.e. on $\R^n$ and with Fatou's lemma, we deduce from \eqref{eqn: formal IBP epsilon xx} that
\[
n\int_H\,(\dett \nabla^2 \vphi)^{1/n} f^\psharp\leq - \psharp \int_H f^{\psharp -1} \n f \cdot S + t
 \int_{\pa H} f^{\psharp} \, d\Hi^{n-1}\,.
\]
Combining this inequality with \eqref{identity}, we complete the proof of the lemma in the case $p\in(1,n)$.

We now consider the case $p=1$. We now have $|Df|(H)<\infty$ and $\spt\,g$ bounded. Thanks to the latter property, by arguing as in \cite[pg. 96]{MV2005} we can assume that $S=T-t\,e_1\in (BV_{{\rm loc}}\cap L^\infty)(H;\R^n)$. Setting $f_k=1_{B_k}\,\min\{f,k\}$, $k\in\N$, then $f_k\,S\in BV(\R^n;\R^n)$ and by the divergence theorem
\[
\Div(f_k\,S)(H)=\int_{\pa H}f_k\,S\cdot(-\ee)=\int_{\pa H}f_k\,T\cdot(-\ee)+t\int_{\pa H}f\le t\int_{\pa H}f\,.
\]
If we identify $f_k$ and $S$ with their precise representatives, we have
\[
\Div(f_k\,S)(H)=\int_H\,f_k\,d(\Div S)+\int_H\,S\cdot Df_k
\]
where, of course,
\[
\int_H\,f_k\,d(\Div S)=\int_H\,f_k\,d(\Div T)\ge n\,\int_H\,f_k\,(\det\nabla^2\vphi)^{1/n}\,.
\]
We have thus proved
\begin{equation}
  \label{second x}
  n\,\int_H\,f_k\,(\det\nabla^2\vphi)^{1/n}\le-\int_H\,S\cdot Df_k+t\,\int_{\pa H}\,f_k\,d\H^{n-1}\,.
\end{equation}
By monotone convergence $\int_{\pa H}f_k\to\int_{\pa H}f$, while \eqref{eqn: supports} and the boundedness of $\spt g$ imply the existence of $R>0$ such that $|S|\le R$ on $\spt(Df)$, and thus
\[
\Big|\int_H\,S\cdot Df_k-\int_H\,S\cdot Df_k\Big|\le R\,|Df|\Big(H\setminus( B_k\cup\{f<k\}^{(1)}\Big)
\]
where $E^{(1)}$ denotes the set of density points of a Borel set $E\subset\R^n$ and we have used $D(1_E\,f)(K)=Df(E^{(1)}\cap K)$ for every $K\subset\R^n$. Since $|Df|(H)<\infty$, letting $k\to\infty$ and finally exploiting Fatou's lemma we deduce from \eqref{second x}
\[
-\int_H\,S\cdot Df+t\,\int_{\pa H}\,f\,d\H^{n-1}\ge   n\,\int_H\,f\,(\det\nabla^2\vphi)^{1/n}=n \int_H\,g\,,
\]
where in the last inequality we have used \eqref{identity}. The proof is complete.
\end{proof}

\begin{proof}[Proof of Theorem~\ref{thm: equality cases}] Let us consider two functions $f$ and $g$ as in Lemma \ref{lem: IBP} such that, for some $t\in\R$,
\begin{equation}
  \label{equality}
  n\int_{H}g^\psharp=\psharp\|\nabla f\|_{L^p(H)}\,Y(t,g)+t\,\int_{\pa H}f^\psharp\qquad\mbox{with}\quad\int_{\pa H}f^\psharp>0\,.
\end{equation}
where $|Df|(H)$ replaces $\|\nabla f\|_{L^p(H)}$ if $p=1$.
By arguing as in the proof of \cite[Proposition 6]{cenv} in the case $p\in(1,n)$, and as in \cite[Theorem A.1]{FigalliMaggiPratelliINVENTIONES} if $p=1$, we find that $T(x)=\n \vphi(x) = \lambda (x-x_0)$ for some $\lambda >0$ and $x_0 \in \R^n$.

We claim that $x_0 \cdot\ee=0$. Keeping the proof of Lemma~\ref{lem: IBP} in mind, \eqref{equality} implies that
\[
\lim_{\e\to 0^+} \lim_{k \to \infty} \int_{\pa \He} (T \cdot \ee) \,f_{\e,k}^{\psharp}\, d\Hi^{n-1}=0\,,
\]
where $T= \lambda (x-x_0)$ gives
\[
\int_{\pa \He} (T \cdot \ee) \,f_{\e,k}^{\psharp}=\lambda(\e-x_0 \cdot \ee)\int_{\pa \He} f_{\e,k}^{\psharp}\,.
\]
Since we have proved that
\[
\lim_{\e\to 0^+} \lim_{k \to \infty} \int_{\pa \He}f_{\e,k}^{\psharp}\, d\Hi^{n-1}=\int_{\pa H}f^{\psharp}\, d\Hi^{n-1}\,,
\]
where the latter quantity is assumed positive, we conclude that $x_0 \cdot \ee=0$, as claimed. Up to a translation and up to apply an $L^\pstar$-norm preserving dilation to $f$, we can now assume that $x_0=0$ and $\lambda=1$, that is $T(x)=x$.

We first consider the case $p\in(1,n)$. By combining \eqref{formal} and \eqref{equality} we find that we have an equality case in the H\"{o}lder's inequality $\int_{\Ha} A \cdot B \, dx \le\|A\|_{L^p(H)}\,\|B\|_{L^{p'}(H)}$ with
\[
A =- \n f\qquad B = f^{\psharp-1} (x- t\,\ee )\,.
\]
In particular, there exist Borel functions $v:H\to\R^n$ and $a,b:H\to[0,\infty)$ such that $A= a\, v$, $B=b\,v$, and $a = c\,b^{1/(p-1)}$ for some constant $c>0$. Hence, if we set $r=|x-t\,\ee|$ and $v=(x-t\,\ee)/r$, there exists a Borel function $u :[0,\infty)\to[0,\infty)$ such that
\[
f(x)=u(r)\qquad
-\n f (x) = -u'(r) \frac{ x - t\,\ee}{|x-t \,\ee|}\,,
\]
and the above conditions hold with $a = -u'(r)$ and  $b= r u(r)^{\psharp -1}$. In particular,
\[
-u'(r) = c \, (r u(r)^{\psharp - 1})^{1/(p-1)}\qquad\mbox{for a.e. $r>0$}\,,
\]
and
consequently, for some $c_1>0$ and $c_2 \in \R$
\[
u(r) =( c_1 r^{p'} + c_2)_+^{-n/\pstar}\qquad\forall r>0\,,
\]
where $x_+=\max\{x,0\}$. In terms of $f$, this means that
\[
f(x) = (c_1|x-\, t \ee|^{p'}+ c_2)_+^{-n/\pstar}\qquad\forall x\in H\,.
\]
The cases where $c_2$ is positive, zero, and negative correspond, respectively, to $f$ being a dilation-translation image of $U_S$, $U_E$, and $U_{BE}$. If $t>0$, the finiteness of the $L^\pstar(H)$-norm of $f$ excludes the possibilities that $f$ is a dilation-translation image orthogonal to $\ee$ of $U_E$ and $U_{BE}$.

Let us now consider the case $p=1$. Recall that we have already set $T(x)=x$, so that $f=g$ and the combination of \eqref{formal} and \eqref{equality} gives
\begin{equation}\label{eqn: holder p=1}
-\int_{\Ha} (x -t\,\ee )\cdot Df = \| \cdot - t\, \ee\|_{L^\infty (\supp(Df))}  |Df|(H)\,,
\end{equation}
that is
\[
-Df = \frac{x - t\, \ee}{|x- t\, \ee|} |Df| \qquad\mbox{as measures on $H$}\,.
\]
By \cite[Exercise 15.19]{maggiBOOK}, there exists $\mu>0$ such that $f=c\,1_{H\cap B_\mu(t\,\ee)}$, as required.
\end{proof}

\bibliography{references3}
\bibliographystyle{alpha}

\end{document}